\documentclass[10pt]{article}

\usepackage[colorlinks,linkcolor=blue,citecolor=blue]{hyperref}
\usepackage{amsthm}
\usepackage{multirow}
\usepackage{marginnote}
\usepackage{a4wide}
\usepackage{amssymb}
\usepackage{amsfonts}
\usepackage{amsmath}
\usepackage{mathrsfs}
\usepackage{tikz}
\usetikzlibrary{arrows,matrix}
\usetikzlibrary{positioning}
\usepackage{mdframed} 
\usepackage{lipsum} 
\usepackage{extarrows} 
\usepackage{color}
\usepackage{enumitem} 
\usepackage{tocloft} 
\usepackage{titlesec} 

\usepackage[all]{xy} 

\input xy
\xyoption{arrow} \xyoption{matrix}

\date{}

\newtheorem{proposition}{Proposition}[section]
\newtheorem{theorem}[proposition]{Theorem}
\newtheorem{lemma}[proposition]{Lemma}
\newtheorem{example}[proposition]{Example}
\newtheorem{definition}[proposition]{Definition}
\newtheorem{corollary}[proposition]{Corollary}

\def\der{\partial }

\def\nFM0{{\nu }_{F,M_0}}
\def\nFN0{{\nu }_{F,N_0}}
\def\nGN0{{\nu }_{G,N_0}}

\def\N0{ {\bf N}_0 }

\def\t{\otimes}

\def\ra{\rightarrow}

\def\Xpm{X^{\pm }}

\def\s{\sigma}

\def\l1{{\lambda}_1}

\def\a{\alpha}
\def\a0{ {\alpha }_0}
\def\a1{ {\alpha }_1}

\def\l{\lambda}
\def\o{\omega}

\def\nFGM0{{\nu }_{F,G,M_0}}

\def\nFN0{{\nu}_{F,N_0}}

\def\sm{{\sigma}^m}

\def\sm1{{\sigma}^{-1}}

\def\smtp1{{\sigma}^{-t+1}}

\def\o{\omega }
\def\S1{S^{-1}}

\def\Xpm1{X^{\pm 1}_1}

\def\sPM1{{\sigma }^{\pm 1}}
\def\sMP1{{\sigma }^{\mp 1 }}

\def\b{\beta}

\def\di{{\rm d.ind}}

\def\L{\Lambda}

\def\Ytm1{Y^{t-1}}
\def\Yim1{Y^{i-1}}

\def\ass{{\rm ass}}

\def\CZ{{\cal Z}}

\def\Aut{{\rm Aut}}

\def\ker{ {\rm ker } }

\def\SL2Z{ {\rm SL}_2({\bf Z}) }

\def\CZ{ {\cal Z}}

\def\Gp1{ G^{1 , 1 } }
\def\P11{ P^{-1 , 1 } }
\def\Pp1{ P^{1 , 1 } }

\def\nCLsr{{}^\nu\kern-2pt {\cal L}^{\sigma , \rho  }}
\def\nP{{}^\nu \kern-2pt P}
\def\nL{{}^\nu\kern-2pt L}
\def\nLL{{}^\nu\kern-2pt \Lambda}
\def\nPsr{{}^\nu\kern-2pt P^{\sigma , \rho  }}
\def\nLsr{{}^\nu\kern-2pt L^{\sigma , \rho  }}
\def\nuCL{{}^\nu\kern-2pt  {\cal L}}
\def\nCLsr{{}^\nu\kern-2pt {\cal L}^{\sigma , \rho  }}
\def\nCL1m{{}^\nu\kern-2pt {\cal L}^{-1 , 1  }}
\def\x1nu{x^\frac{1}{\nu}}
\def\xm1nu{x^{-\frac{1}{\nu}}}

\def\rad{{\rm rad}}

\def\ra{\rightarrow }

\def\CB{{\cal B}}

\def\CI{{\cal I}}

\def\CC{ {\cal C}}

\def\nAM0{{\nu }_{{\cal A},M_0}}
\def\nAN0{{\nu }_{{\cal A},N_0}}

\def\bR{\overline{R}}

\def\ga{\mathfrak{a}}
\def\gb{\mathfrak{b}}
\def\gc{\mathfrak{c}}

\def\gn{\mathfrak{n}}

\def\gp{\mathfrak{p}}
\def\gq{\mathfrak{q}}

\def\SL{{\rm SL}}

\def\Spec{{\rm Spec}}

\def\di!{\frac{\der^i}{i!}}
\def\dik!{\frac{\der^k_i}{k!}}

\def\gl{\mathfrak{l}}

\def\id{{\rm id}}

\def\N{\mathbb{N}}

\def\0{\overline{0}}
\def\1{\overline{1}}

\def\Ln1{\L_{n,\overline{1}}}

\def\a1{a_{\overline{1}}}

\def\bs{\overline{s}}

\def\S{\Sigma}

\def\vn1{\overrightarrow{n-1}}

\def\im{{\rm im}}

\def\gl{{\rm gl}}
\def\sl{{\rm sl}}

\def\mJ{\mathbb{J}}
\def\mI{\mathbb{I}}

\def\ann{{\rm ann}}
\def\lann{{\rm l.ann}}
\def\rann{{\rm r.ann}}

\def\bgp{\overline{\gp}}

\def\K1{{\rm K}_1}

\def\hmI1{\widehat{\mI_1}}
\def\tmI1{\widetilde{\mI_1}}
\def\tmJ1{\widetilde{\mJ_1}}
\def\hB1{\widehat{B_1}}
\def\hCB1{\widehat{\CB_1}}

\def\bS{\overline{S}}

\def\Den{{\rm Den}}

\def\Ore{{\rm Ore}}

\def\Den{{\rm Den}}

\def\Loc{{\rm Loc}}

\def\Ass{{\rm Ass}}

\def\maxDen{{\rm max.Den}}

\def\br{\overline{r}}

\def\bs{\overline{s}}

\def\ga{\mathfrak{a}}

\def\tor{{\rm tor}}

\def \S{\mathcal{S}}

\def\sl2{\mathfrak{sl}_2}

\def\sl2{\mathfrak{sl}_2}
\def\gl2{\mathfrak{gl}_2}

\def\b1{\overline{1}}

\def\Denlip{{\rm Den}_{l, ip}}

\def\RSm{ R\langle S^{-1}\rangle}

\def\gl{{\mathfrak{l}}}

\def\tbR{\widetilde{\overline{R}}}
\def\tbS{\widetilde{\overline{S}}}
\def\tbgp{\widetilde{\bgp}}

\makeatletter
\newenvironment{proof*}[1][\proofname]{\par
  \pushQED{\qed}%
  \normalfont \partopsep=\z@skip \topsep=\z@skip
  \trivlist
  \item[\hskip\labelsep
        \itshape
    #1\@addpunct{.}]\ignorespaces
}{%
  \popQED\endtrivlist\@endpefalse
}
\makeatother

\begin{document}

\author{V. V. \  Bavula  
}
\title{The minimal primes of localizations of rings }

\maketitle

\begin{abstract}
The set of minimal primes of a ring is a very  important set as far the  spectrum of a ring is concerned as every prime contains a minimal prime. So, knowing the minimal primes is the first (important and difficult) step in describing the spectrum. 
In the algebraic geometry, the minimal primes of the algebra of regular functions on an algebraic variety determine/correspond to the   irreducible components of the variety.  The aim of the paper is to obtain several descriptions of the set of minimal prime ideals of localizations of rings under several natural assumptions. In particular, the following case are considered: a  localization of a semiprime ring with finite set of minimal primes; a  localization of a prime rich ring  where the localization respects the ideal structure of primes and primeness of certain minimal primes; a localization of a ring at a left denominator set generated by normal elements, and others. As an application, for a semiprime ring with finitely many minimal primes, a description of the minimal primes of it's largest left/right quotient ring is obtained. 
 
  For  a semiprime ring $R$ with finitely many minimal primes $\min (R)$, 
 criteria are given for the map $$\rho_{R,\min}: \min (R)\ra \min (Z(R)), \;\; \gp\mapsto \gp\cap Z(R)$$  being a well-defined map or  surjective  where $Z(R)$ is the centre of $R$.

$\noindent$

{\em Key words: semiprime ring, prime ring, localization, classical left  quotient ring, largest left quotient ring, Ore set, denominator set, minimal primes, prime rich ring, denominator set that respects the ideal structure of a ring.  }\\


{ \small \tableofcontents}
\end{abstract}


\section{Introduction} \label{INTR} 

{\bf Notation.} In this paper, module means a left module, and the following notation is fixed:
\begin{itemize}
\item  $R$ is a ring  with  1, $R^\times$ is its group of units, 
$Z(R)$ is the centre of $R$, 
 $\gn_R$ is the  prime  radical and $\rad (R)$ is the  radical  of $R$,  and $\min (R)$ is the  set  of minimal  primes of $R$;

\item $\Spec (R)$ is the prime spectrum and $\Spec_c (R)$ is the completely prime spectrum  of $R$; 

 \item $\Den_l(R, \ga )$ is the set of left denominator sets $S$ of $R$ with $\ass_l (S)=\ga$ where $\ga$ is an ideal of $R$ and $\ass_l (S):= \{r\in R\, | \, sr=0$ for some $s\in S\}$;
 
\item For $S\in \Den_l(R, \ga)$, $\min (R, S):=\{ \gp\in \min (R)\, | \, \gp\cap S=\emptyset\}$ and $\min (R, S, \id):=\{ \gp\in \min (R)\, | \, S^{-1}\gp S^{-1}R\neq S^{-1}R\}$;
 
\item $\maxDen_l(R)$ is the set of maximal left denominator sets of $R$ (it is  a {\em non-empty} set, \cite{larglquot});

\item   $\CC_R$  is the  set of  regular  elements of the ring $R$ (i.e. $\CC_R$ is the  set of non-zero-divisors of the ring $R$);

\item   $Q_{l,cl}(R):= \CC_R^{-1}R$  is the  {\em left quotient ring}   (the {\em  classical left   ring of fractions}) of the ring $R$ (if  it  exists);

\item $S_{l,\ga }(R)$ is the {\em largest element} of the partially ordered set (poset) $(\Den_l(R, \ga ),
\subseteq )$ and $Q_{l,\ga }(R):=S_{l,\ga }(R)^{-1} R$ is  the
{\em largest left quotient ring associated to} $\ga$, $S_{l,\ga }(R) $ exists, \cite[Theorem 2.1.(2)]{larglquot};

\item In particular,
$S_l(R)=S_{l,0}(R)$ is the largest element of the poset
$(\Den_l(R, 0), \subseteq )$ and $Q_l(R):=S_0^{-1}R$ is the
largest left quotient ring of $R$; 

\item For a   commutative ring $R$  and $\gp \in \Spec (R)$,   $R_\gp:=\Big(R\backslash \gp \Big)^{-1}R$ is the localization of $R$ at the prime ideal $\gp$. 


       
\end{itemize}

{\bf Semiprime ideals and minimal primes.} Prime and semiprime rings are large and important classes of rings both in commutative and non-commutative algebra. They also have geometric flavour. The algebra of regular functions on an affine algebraic variety is a semiprime algebra and the corresponding  ideal of definition of the variety is a semiprime ideal. If, in addition, the variety is irreducible then  the algebra of regular functions on it is a prime algebra and the defining ideal  is a prime ideal. In the algebraic geometry, the minimal primes of the algebra of regular functions correspond to the irreducible components of the algebraic variety and  generators of  minimal primes are the equations of the corresponding irreducible component of the variety. In the commutative situation, 
 localizations respect many operations and concepts. This is not the case in the noncommutative setting (in general, a localization of an ideal is not an ideal, a localization of a prime ideal is not a prime ideal, there is little  connection between prime ideals of a ring and its localization, etc).

  Let us describe the main results and the structure of the paper. 
 At the beginning of Section \ref{MINPRLOCSPRIMES}, we recall some results on localizations that are used in the paper. Suppose  that $\ga$ is an ideal of a ring $R$. In general, the left ideal  $S^{-1}\ga$ of the ring  $S^{-1}R$   is {\em not} an ideal. 
 Proposition \ref{A11Sep23} is a criterion for the left ideal $S^{-1}\ga$ being an ideal of $S^{-1}R$. Corollary \ref{a6Oct23} presents criteria for a localization of a prime ideal being a prime ideal (under certain conditions).
 \\

{\bf The set $\min (S^{-1}R)$ where $R$ is a prime rich ring and $S$ respects the ideal structure of primes.}  A ring is called a {\bf prime rich ring} if each  ideal contains a finite product of prime ideals that contain it.
 Proposition \ref{aA29Sep23} is a criterion for a ring being prime rich where $\min(\ga)$ is the set of minimal primes over an ideal $\ga$ (its proof is given in Section \ref{MINPRLOCSPRIMES}) .

\begin{proposition}\label{aA29Sep23}
For a ring $R$ the following conditions are equivalent:
\begin{enumerate}
\item The ring $R$ is a prime rich ring.

\item Every ideal $\ga$ of $R$ contains a finite product of its minimal primes. 

\item For every ideal $\ga$ of $R$, $|\min (\ga)|<\infty$ and the prime radical $\gn_{R/\ga}$ of the factor ring $R/\ga$ is a nilpotent ideal.

\end{enumerate}
\end{proposition} 
  Every one-sided Noetherian ring $R$ is prime rich.
  A left denominator set $S$ of a ring $R$ is called {\bf an ideal preserving left denominator set} (or a left denominator set that  {\bf respects the ideal structure} of $R$)  if the map 
  $$S^{-1}: \CI (R)\ra \CI(S^{-1}R), \;\; I\mapsto S^{-1}I$$ is a well-defined map, i.e.  for each  ideal $I$ of $R$, the left ideal $S^{-1}I$ of  of the ring $S^{-1}R$ is an ideal (where $\CI (R)$ is the set of all ideals of $R$).  We say that a left denominator set $S$ of a ring $R$ {\bf respects the ideal structure of primes} of $R$ if for each prime ideal $\gp\in \Spec (R)$ the left ideal  $S^{-1}\gp$ of the ring $S^{-1}R$ is an ideal. All the left denominator sets of a left Noetherian ring are ideal preserving and as a result also respect the ideal structure of primes. 

For a prime rich ring $R$ and its left denominator set $S\in \Den_l(R,\ga)$ that respects the  ideal structure of primes,  Proposition \ref{29Sep23} describes the set of minimal primes of the ring $S^{-1}R$ via the minimal primes of the ring $R$. 
 
\begin{proposition}\label{29Sep23}
Suppose that  $R$ is  a prime rich ring, a left denominator set  $S\in \Den (R, \ga)$ respects the ideal structure of primes and $S^{-1}\gp\in \Spec (S^{-1}R)$ for all $\gp\in \min (R, S)$. 
Then 
\begin{enumerate}

\item  $1\leq |\min (R,S)|\leq |\min (R)|<\infty$,  the set 
$$\min (S^{-1}R)=\min \{ S^{-1}\gp \, | \, \gp \in \min (R,S)\}$$
is a finite set and $|\min (S^{-1}R)|\leq  |\min (R,S)|<\infty$. 

\item $\min (S^{-1}R)= \{ S^{-1}\gp \, | \, \gp \in \min (R,S)\}$ iff the ideals $\{ S^{-1}\gp \, | \, \gp \in \min (R,S)\}$ are incomparable.

 \item If in addition, for each prime ideal $\gp\in \min (R, S)$, the ideal $\s_S^{-1}(S^{-1}\gp)$  is a finitely generated right $R$-module (e.g. $R$ is a right Noetheiran ring) then  $$\min (S^{-1}R)= \{ S^{-1}\gp \, | \, \gp \in \min (R,S)\}$$ where $\s_S: R\ra S^{-1}R$, $r\mapsto \frac{r}{1}$.

\end{enumerate}
\end{proposition}

\begin{corollary}\label{a29Sep23}
Suppose that  $R$ is  a  right Noetherian ring,   $S\in \Den (R, 0)$ and  $S^{-1}R$ is a left Noetherian ring (e.g. $R$ is a Noetherian ring). 
Then $$\min (S^{-1}R)= \{ S^{-1}\gp \, | \, \gp \in \min (R,S)\}.$$
In particular, $\min (R,S)\neq \emptyset$. 
\end{corollary}

A proof of Proposition \ref{29Sep23} and Corollary \ref{a29Sep23} are given in Section \ref{MINPRLOCSPRIMES}. \\

  {\bf Description of the set of minimal primes of a localization of a semiprime ring at a regular denominator set.}  The aim of  Section \ref{DESCRMINPR} is to prove Theorem \ref{A10Sep23}.(2) that provides  an explicit description of  the set  of minimal primes of localizations of  semiprime  rings provided that they  have only finitely many minimal primes (statement 1 is a known result). Theorem \ref{A10Sep23}.(2) is one of the  key facts in  understanding   the structure of  embeddings of localizations of semiprime  rings  into  semisimple or  semiprimary rings, see \cite{Embed-SPrime-SSA}.

\begin{theorem}\label{A10Sep23}
Let $R$ be a semiprime ring and $S\in \Den_l(R, 0)$. Then 
\begin{enumerate}

\item The ring $S^{-1}R$ is a semiprime ring.

\item If, in  addition,  $|\min (R)|<\infty$ then $\min (S^{-1}R)=\{ S^{-1}\gp\, | \, \gp \in \min (R)\}$, i.e. the map $\min (R)\ra \min (S^{-1}R)$, $\gp \mapsto S^{-1}\gp$ is a bijection and $|\min (S^{-1}R)|=|\min (R)|$.

\item If, in  addition,  $|\min (R)|<\infty$ then
$\pi_\gp (S)\in \Den_l(R/\gp , 0)$ and  $$\pi_\gp (S)^{-1}(R/\gp)\simeq S^{-1}R/S^{-1}\gp \simeq S^{-1}(R/\gp)$$ where $\pi_\gp : R\ra R/\gp$, $r\mapsto r+\gp$.
\end{enumerate}
\end{theorem}
 In general, for a ring $R$, its {\em classical left quotient ring} of $R$,  $$Q_{l,cl}(R):=\CC_R^{-1}R, $$ does not exists (e.g. the algebras of polynomial  integro-differential operators,  \cite{Bav-intdifline})  but for each ring $R$ there is the largest left Ore set  $S_l(R)$ of regular elements of $R$ and the ring $$Q_l(R):=S_l(R)^{-1}R$$ is called the {\em  largest left quotient ring} of $R$, \cite{Bav-intdifline, larglquot}. 
 As a corollary, for a semiprime ring $R$ with finitely many minimal primes, Corollary \ref{c10Sep23} describes the set of minimal primes of the largest left quotient ring $Q_l(R)$ of $R$. 
 
 Corollary \ref{aA10Sep23} is a generalization of Theorem \ref{A10Sep23} to the case when the denominator set $S$ contains zero divisors of $R$. To be more precisely, $S\in \Den_l(R,\ga)$ where  $\ga\in \Spec (R)$.
 
  Theorem \ref{28Sep23} is a generalization of Theorem \ref{A10Sep23}.(2,3) for left denominator sets that contain zero divisors.   
  
   Proposition \ref{b28Sep23}, which is a corollary of Theorem \ref{28Sep23},  describes the set of  minimal primes, $\min (S^{-1}R)$,  of a localization of a semiprime ring $R$ under  conditions that hold for all commutative semiprime rings (namely, for all  $\gp \in \min (R, S)$, $\gp$ is a completely prime ideal of $R$ and $S^{-1}\gp$ is an ideal of $S^{-1}R$).
  
\begin{proposition}\label{b28Sep23}
Let $R$ be a semiprime  ring with $|\min (R)|<\infty$ and $S\in \Den (R, \ga)$. Suppose that for all $\gp \in \min (R, S)$, $\gp$ is a completely prime ideal of $R$ and $S^{-1}\gp$ is an ideal of $S^{-1}R$. Then 
\begin{enumerate}

\item  The ring $S^{-1}R$ is a semiprime ring with $\min (S^{-1}R)=\{ S^{-1}\gp\, | \, \gp \in \min (R,S)\}\subseteq \Spec_c(S^{-1}R)$.

\item  In statement 1,  all ideals in the set $\min (S^{-1}R)$ are distinct, i.e. $|\min (S^{-1}R)|=|\min (R,S)|$.  

\end{enumerate}
\end{proposition} 
Proofs of Theorem \ref{A10Sep23} and Proposition \ref{b28Sep23} are given in Section \ref{DESCRMINPR}.\\

{\bf Description of minimal primes of localizations of rings at multiplicative sets generated by normal elements.}  
Inversion of elements in a ring  is an  important and difficult operation that  `simplifies'  the situation as a rule.   Ore's method of localization is an example of a theory of {\em one-sided fractions} where by definition only the elements of a {\em denominator set} can be inverted and the result is a ring of one-sided fractions which always exsists, i.e. it is not equal to zero, \cite{Stenstrom-RingQuot,Jategaonkar-LocNRings,MR}. In \cite[Theorem 4.15]{Bav-intdifline}, it is proven that the elements of an {\em arbitrary} (left and right) {\em Ore set} can be inverted and the result  is also a ring of one-sided fractions (which always exists)  but in general Ore set is not a denominator set. This was the starting point in \cite{LocSets} and \cite{Loc-groupUnits} where the  most general theory of one-sided fractions was presented.

\begin{theorem}\label{10Jan19-st1,2}
(\cite[Theorem 1.6.(1,2)]{LocSets})  Let $R$ be a ring and $S\in \Ore (R)$.
\begin{enumerate}
\item $\ga :=\{ r\in R\, | \, srt=0$ for some elements $s,t\in S\}$ is an ideal of $R$ such that $\ga \neq R$. 
\item  Let $\pi : R\ra \bR :=R/\ga$, $r\mapsto \br =r+\ga$. Then $\bS :=
\pi (S) \in \Den (\bR , 0)$, $\ass_R(S)=\ga $, $S$ is a  localizable set,  and $\RSm \simeq \bS^{-1}\bR$, an $R$-isomorphism  (the ring $\RSm$ is defined in (\ref{IRbSbm})).
\end{enumerate}
\end{theorem}

 Let $R$ be a ring and $S$ be a multiplicative subset of $R$ which is generated by normal elements of $R$. Recall that an element $x\in R$ is called a {\em normal element} of $R$ if $Rx=xR$.  Clearly,  $S\in \Ore (R)$ is an  Ore set of $R$. By Theorem \ref{10Jan19-st1,2}, $S$ is a localizable set of $R$ such that  $\ass_R(S)=\ga $ and $\RSm \simeq \bS^{-1}\bR$, see Section  \ref{MINPTNORM} for details.   We keep the notation of Theorem \ref{10Jan19-st1,2}. Let  $$\widetilde{\pi}: \bR\ra \widetilde{R}:=\bR/  \gn_{\bR}, \;\;\br\mapsto \widetilde{r}=\widetilde{\br}=\br +\gn_{\bR}$$ where $\gn_{\bR}$ is the prime radical of the ring $\bR$. There are bijections
$$\min (\ga)\stackrel{\pi}{\ra} \min (\bR)\stackrel{\widetilde{\pi}}{\ra}\min(\widetilde{R}), \;\;  \gp\mapsto \bgp \mapsto \widetilde{\overline{\gp}}.$$
Let $\min (\ga , S):=\{ \gp \in \min (\ga)\, | \, \gp\cap S =\emptyset\}$.

Theorem \ref{A2Oct23} describes the set of minimal primes of the ring $\RSm$. In general, the ring $\RSm$ is not a semiprime ring but precisely  the same result as Theorem \ref{A10Sep23} holds for it but without the `semiprimeness' assumption, see Theorem \ref{A2Oct23}.(4a). The key ideal of the proof of this result is to reduce it to the semiprime case by considering the semiprime ring
$\RSm/\gn_{\RSm}$.

\begin{theorem}\label{A2Oct23}
Let $R$ be a ring and $S$ be a multiplicative subset of $R$ which is generated by normal elements and such that $|\min (\ga)|<\infty$ where $\ga$ is defined in Theorem \ref{10Jan19-st1,2}.(1), e.g. $R$ is a one-sided Noetherian ring.  For each $\gp \in \min (\ga )$, let  $\pi_\gp : R\ra R/\gp$, $r\mapsto r+\gp$. Then 
\begin{enumerate}

\item The map 
$\min (\ga)\ra \Spec (\RSm)=\Spec (\bS^{-1}\bR)$, $\gp \mapsto \bS^{-1}\bgp $ is an injection and $\gn_{\bS^{-1}\bR}\subseteq \bS^{-1}\gn_{\bR}$. 

\item $\widetilde{\bS}\in \Den (\widetilde{R},0)$,   $\bS^{-1}\bR/\bS^{-1}\gn_{\bR}\simeq
\widetilde{\bS}^{-1}\widetilde{R}$, and  the map $\min (\ga)\ra \min (\tbS^{-1}\widetilde{R})$, $\gp\mapsto \tbS^{-1}\tbgp$ is a bijection.

\item  For each $\gp \in \min (\ga )$, $\pi_\gp (S)\in \Den(R/\gp, 0)$ and $\pi_\gp (S)^{-1}(R/\gp)\simeq \widetilde{\bS}^{-1}\widetilde{R}/\widetilde{\bS}^{-1}\widetilde{\bgp}\simeq \bS^{-1}\bR/\bS^{-1}\bgp$ and there is a  commutative diagram:

$$\begin{array}[c]{cccccccc} 
\pi_\gp:& R & \stackrel{\pi}{\ra} & \bR &  \stackrel{\widetilde{\pi}}{\ra}& \widetilde{R} & \stackrel{\widetilde{\pi}_\gp}{\ra} & R/\gp  \\
 & & &\downarrow & &\downarrow & &\downarrow   \\
& & &\bS^{-1}\bR &\ra  &\widetilde{\bS}^{-1}\widetilde{R} &\ra & \pi_\gp(S)^{-1}(R/\gp)  \\
 & & &\downarrow  &\nearrow_{\simeq} &\downarrow  &  {\nearrow}_{\simeq}&   \\
 &\bS^{-1}\bR/\gn_{\bS^{-1}\bR} &\ra &\bS^{-1}\bR/\bS^{-1}\gn_{\bR} &\ra &\widetilde{\bS}^{-1}\widetilde{R}/\widetilde{\bS}^{-1}\widetilde{\bgp} & &  \\
& & & \downarrow& {\nearrow}_{\simeq}& & &   \\ 
& & &\bS^{-1}\bR/\bS^{-1}\gp & & & &   \\ 
\end{array}$$

\item Suppose that the prime radical $\gn_{\bR}$ is a nilpotent ideal (e.g. the ring $\bR$ is a one-sided Noetherian). Then 

\begin{enumerate}

\item The map 
$\min (\ga)\ra \min (\RSm)=\min(\bS^{-1}\bR)$, $\gp \mapsto \bS^{-1}\bgp $ is a bijection, i.e.  $ \min (\RSm)=\{   \bS^{-1}\bgp\, | \, \gp \in \min (\ga)\}$. 

\item $\widetilde{\bS}\in \Den (\widetilde{R},0)$,  $\bS^{-1}\gn_{\bR}=\gn_{\bS^{-1}\bR}$ and  $\bS^{-1}\bR/\bS^{-1}\gn_{\bR}\simeq
\bS^{-1}\bR/\gn_{\bS^{-1}\bR}\simeq
\widetilde{\bS}^{-1}\widetilde{R}$.


\end{enumerate} 
\end{enumerate}
\end{theorem}

The proof of Theorem \ref{A2Oct23} is given in Section  \ref{MINPTNORM}. 
 
  Lemma \ref{a5Oct23} is a generalization  of Theorem \ref{A2Oct23} to a more general situation. \\
  
{\bf  Minimal primes of a semiprime ring and of its centre.}  For a ring $R$,  the {\bf restriction map}
\begin{equation}\label{SpecRZR}
\rho_R: \Spec (R)\ra \Spec (Z(R)), P\mapsto P\cap Z(R)
\end{equation}
is a well-defined map. In general,   the  map
\begin{equation}\label{rminZR}
\rho_{R,\min}: \min (R)\ra \min (Z(R)), \;\; \gp\mapsto \gp\cap Z(R)
\end{equation}
is not a  well-defined map, see Lemma \ref{b29Sep23}. For each natural number $n\geq 1$, Lemma \ref{b29Sep23} provides an example of a finitely generated algebra $A_n$ over a field $K$ such that $A_n$  is a domain,  the   centre  $Z(A_n)=P_n=K[z_1, \ldots , z_n]$ is polynomial algebra in $n$ variables and   $$\rho_{A_n} (\min (A_n))=\{ \{ 0\}, \gp_I:=(x_i)_{i\in I}\, | \, \emptyset \neq I\subseteq \{ 1, \ldots , n\}\}.$$ The set   $\rho_{A_n} (\min (A_n))$ contains prime ideals of all possible   heights.  

 Lemma \ref{b29Sep23} provides an example of a countably  generated algebra $A_\infty$ over a (countable)  field $K$ such that the set $\min (A_\infty)$  is uncountable and every minimal prime of $A_\infty$ is a completely prime ideal. Notice that if  the field $K$ is a countable field then the algebra $A_\infty$ contains a countably many elements.

 For a semiprime ring $R$ with $|\min (R)|<\infty$, Proposition   \ref{aB25Sep23} is a criterion for the  map $\rho_{R,\min}$
 being a  well-defined map.

\begin{proposition}\label{aB25Sep23}
Let $R$ be a semiprime ring with $|\min (R)|<\infty$.  Then the following statements  are equivalent: 
\begin{enumerate}

\item $\CC_{Z(R)}\subseteq \CC_R$.

\item $\CC_{Z(R)}\cap \gp=\emptyset$  for all $\gp \in \min (R)$.

\item The map $\rho_{R,\min}$ is a well-defined map. 

\end{enumerate}
\end{proposition}

\begin{proof} Proposition \ref{aB25Sep23} is equivalent to Proposition \ref{C25Sep23}.
\end{proof}

Proposition \ref{B25Sep23} is a criterion for the map $\rho_{R,\min}$  being  a surjection.

\begin{proposition}\label{B25Sep23}
Let $R$ be a semiprime ring with $|\min (R)|<\infty$ and $|\min (Z(R))|<\infty$.  Then the following statements  are equivalent: 
\begin{enumerate}

\item $\CC_{Z(R)}\subseteq \CC_R$.

\item $\CC_{Z(R)}\cap \gp=\emptyset$  for all $\gp \in \min (R)$.

\item The map $\rho_{R,\min}$  is a well-defined map. 

\item The map $\rho_{R,\min}$  is a surjection. 

\end{enumerate}
\end{proposition}

  Proofs of Proposition  \ref{aB25Sep23} and  Proposition \ref{B25Sep23} are given in Section  \ref{MINPZR}. \\


\section{Descriptions of minimal primes of localizations of  semiprime  rings} \label{MINPRLOCSPRIMES} 

At the beginning of the section, we recall some results on localizations that are used in the paper. Suppose  that $\ga$ is an ideal of a ring $R$. In general, the left ideal  $S^{-1}\ga$ of the ring  $S^{-1}R$   is {\em not} an ideal. 
 Proposition \ref{A11Sep23} is a criterion for the left ideal $S^{-1}\ga$ being an ideal of $S^{-1}R$. Corollary \ref{a6Oct23} presents criteria for a localization of a prime ideal being a prime ideal (under certain conditions). 
For a ring $R$ and its left denominator set $S\in \Den_l(R)$,  Proposition \ref{29Sep23} describes the set of minimal primes of the ring $S^{-1}R$ via the minimal primes of the ring $R$ provided some general conditions hold.  \\

 {\bf The largest regular left Ore set and the largest left quotient ring of a ring}. Let $R$ be a ring. A {\bf multiplicatively closed subset} $S$ of $R$ or a {\bf
 multiplicative subset} of $R$ (i.e., a multiplicative sub-semigroup of $(R,
\cdot )$ such that $1\in S$ and $0\not\in S$) is said to be a {\bf
left Ore set} if it satisfies the {\bf left Ore condition}: for
each $r\in R$ and $s\in S$, 
  $$Sr\bigcap Rs\neq \emptyset.$$
Let $\Ore_l(R)$ be the set of all left Ore sets of $R$.
  For  $S\in \Ore_l(R)$, $\ass_l (S) :=\{ r\in
R\, | \, sr=0 \;\; {\rm for\;  some}\;\; s\in S\}$  is an ideal of
the ring $R$.

A left Ore set $S$ is called a {\bf left denominator set} of the
ring $R$ if $rs=0$ for some elements $ r\in R$ and $s\in S$ implies
$tr=0$ for some element $t\in S$, i.e. $r\in \ass (S)$. Let
$\Den_l(R)$ be the set of all left denominator sets of $R$. For
$S\in \Den_l(R)$, the ring $$S^{-1}R=\{ s^{-1}r\, | \, s\in S, r\in R\}$$
 is called  the {\bf left localization} of the ring $R$ at $S$ (the {\bf
left quotient ring} of $R$ at $S$). In Ore's method of localization one can localize  precisely at left denominator sets.
 In an obvious way, the {\bf right Ore condition}, {\bf  right Ore} and {\bf right  denominator sets} of $R$ are defined and they are denoted by $\Ore_r(R)$ and $\Den_r(R)$, respectively. For
$S\in \Den_r(R)$, the ring $$RS^{-1}=\{ rs^{-1}\, | \, s\in S, r\in R\}$$
 is called  the {\bf right localization} of the ring $R$ at $S$ (the {\bf
right quotient ring} of $R$ at $S$). For  $S\in \Ore_r(R)$, $\ass_r (S) :=\{ r\in
R\, | \, rs=0 \;\; {\rm for\;  some}\;\; s\in S\}$  is an ideal of
the ring $R$.
 
 A left and right Ore or denominator set is  called  an {\bf Ore set}  or a {\bf denominator set} and these sets are denoted by $\Ore (R)$ and $\Den (R)$, respectively. If $S\in \Den (R)$ then $$S^{-1}R\simeq RS^{-1}.$$ 
For an ideal $\ga$ of $R$, let $\Den_*(R,\ga):=\{ S\in \Den_*(R)\, | \, \ass_* (S)=\ga\}$ where $*\in \{ l,r,\emptyset\}$. 

In general, the set $\CC_R$ of regular elements of a ring $R$ is
neither left nor right Ore set of the ring $R$ and as a
 result neither left nor right classical  quotient ring ($Q_{l,cl}(R):=\CC_R^{-1}R$ and
 $Q_{r,cl}(R):=R\CC_R^{-1}$) exists.
 There  exists the largest (w.r.t. $\subseteq$)
 regular left Ore set $S_l(R)$, \cite{larglquot}. This means that the set $S_l(R)$ is an Ore set of
 the ring $R$ that consists
 of regular elements (i.e. $S_l(R)\subseteq \CC_R$) and contains all the left Ore sets in $R$ that consist of
 regular elements. Also, there exists the largest regular  right (respectively, left and right) Ore set  $S_r(R)$ (respectively, $S_r(R)$, $S_{l,r}(R)$) of the ring $R$.
 In general,  the sets $\CC_R$, $S_l(R)$, $S_r(R)$ and $S_{l,r}(R)$ are distinct, for example,
 when $R= \mI_1= K\langle x, \der , \int\rangle$  is the ring of polynomial integro-differential operators  over a field $K$ of characteristic zero,  \cite{Bav-intdifline}. In  \cite{Bav-intdifline},  these four sets are explicitly described  for $R=\mI_1$.

\begin{definition} (\cite{Bav-intdifline, larglquot}.)    The ring
$$Q_l(R):= S_l(R)^{-1}R$$ (respectively, $Q_r(R):=RS_r(R)^{-1}$ and
$Q(R):= S_{l,r}(R)^{-1}R\simeq RS_{l,r}(R)^{-1}$) is  called
the {\bf largest left} (respectively, {\bf right and two-sided})
{\bf quotient ring} of the ring $R$.
\end{definition}
 In general, the rings $Q_l(R)$, $Q_r(R)$ and $Q(R)$
are not isomorphic, for example, when $R= \mI_1$, \cite{Bav-intdifline}.

Let $R$ be a ring. We say that two left localizations of $R$  are {\bf equal} and write $S^{-1}R = S'^{-1}R$ if the map $ S^{-1}R \ra S'^{-1}R$, $s^{-1}r\mapsto s^{-1}r$, is a well-defined isomorphism. This isomorphism is an $R$-isomorphism.   Then the map $S'^{-1}R\ra S^{-1}R$, $s'^{-1}r\mapsto s'^{-1}r$, is also a ring  $R$-isomorphism. So, two localizations are equal iff there is an $R$-{\em isomorphism} between them. So, the relation of `equality' is an equivalence relation on the set of left localizations of the ring $R$. The set of all the equivalence classes is denoted by $\Loc_l(R)$.
Clearly, $S^{-1}R = S'^{-1}R$ iff $\ass_l (S)=\ass_l (T)$, $ \frac{s}{1}\in (S'^{-1}R)^\times$ for all $s\in S$ and $ \frac{s'}{1}\in (S^{-1}R)^\times$ for all $s'\in S'$.

 The next
theorem gives various properties of the ring $Q_l(R)$. In
particular, it describes its group of units.

\begin{theorem}\label{4Jul10}
(\cite[Theorem 2.8]{larglquot}.)
\begin{enumerate}
\item $ S_l (Q_l(R))= Q_l(R)^\times$ {\em and} $S_l(Q_l(R))\cap R=
S_l(R)$.
 \item $Q_l(R)^\times= \langle S_l(R), S_l(R)^{-1}\rangle$, {\em i.e., the
 group of units of the ring $Q_l(R)$ is generated by the sets
 $S_l(R)$ and} $S_l(R)^{-1}:= \{ s^{-1} \, | \, s\in S_l(R)\}$.
 \item $Q_l(R)^\times = \{ s^{-1}t\, | \, s,t\in S_l(R)\}$.
 \item $Q_l(Q_l(R))=Q_l(R)$.
\end{enumerate}
\end{theorem}

 A ring $R$ has {\bf
finite left rank} (i.e. {\bf finite left uniform dimension}) if
there are no infinite direct sums of nonzero left ideals in $R$. For a non-empty subset $T$ of a ring $R$, the sets 
$$\lann_R(T):=\{ r\in R\, | \, rT=0\}\;\; {\rm  and}\;\;\rann_R(T):=\{ r\in R\, | \, Tr=0\}$$
 are called the {\bf left} and {\bf right annihilators} of $T$, respectively. They are left and right ideals of $R$, respectively. If they coincide, we  write $\ann_R(T)$ for their common value. A ring $R$ is called a {\bf left (right) Goldie ring} if $R$ has finite left (right) uniform dimension and $R$ satisfies the ascending chain condition (the a.c.c.) on left (right) annihilators.

The next theorem is a semisimplicity criterion for the ring
$Q_l(R)$  (statements 2-5 are  Goldie's
Theorem).

\begin{theorem}\label{5Jul10}
(\cite[Theorem 2.9]{larglquot}.) The following properties of a ring $R$ are equivalent:
\begin{enumerate}
\item  $Q_l(R)$ is a semisimple ring. \item $Q_{l, cl}(R)$  exists
and is a semisimple ring. \item $R$ is a left order in a
semisimple ring. \item $R$ has finite left rank, satisfies the
ascending chain condition on left annihilators and is a semi-prime
ring. \item A left ideal of $R$ is essential iff it contains a
regular element.
\end{enumerate}
If one of the equivalent conditions holds then $S_0(R) = \CC_R$ and
$Q_l(R) = Q_{l,cl}(R)$.
\end{theorem}

Let $R$ be a  ring,  $S\in \Den_l(R, \ga)$, $$\pi_\ga :R\ra R/\ga, \;\;r\mapsto  \br :=r+\ga \;\; {\rm and}\;\;\s_S:R\ra S^{-1}R, \;\;r\mapsto \frac{r}{1}.$$
 Let $\CI_l(R)$ and $\CI (R)$ be the sets of left and two-sided ideals of the ring $R$, respectively. Then the maps below are well-defined:
\begin{equation}\label{SCICI}
S^{-1}: \CI_l(R)\ra \CI_l(S^{-1}R), \;\; I\mapsto S^{-1}I\;\; {\rm and}\;\; \s_S^{-1}: \CI_l(S^{-1}R)  \ra    \CI_l(R), \;\; J\mapsto \s_S^{-1}(J).
\end{equation}
Furthermore, $S^{-1}\s_S^{-1}(J)=J$ for all $J\in \CI_l(S^{-1}R)$. 
In general, the map $S^{-1}: \CI (R)\ra \CI(S^{-1}R), \;\; I\mapsto S^{-1}I$ is {\em not} defined (see \cite[Example 10L]{GW}) but the map 
\begin{equation}\label{SCICI-1}
\s_S^{-1}:\CI (S^{-1}R)\ra   \CI(R), \;\; J\mapsto \s_S^{-1}(J)
\end{equation}
is a well-defined map and  $S^{-1}\s_S^{-1}(J)=J$ for all $J\in \CI(S^{-1}R)$. So, the localized ring $S^{-1}R$ has `less' ideals than  expected since the images of some ideals of the ring $R$ are only {\em left} ideals but not two-sided ideals of the ring $S^{-1}R$. This phenomenon does not occur if the ring $R$ is a commutative ring or the denominator set $S$ consists of central elements of $R$. Nonetheless all the ideal of $S^{-1}R$ can be obtained from some ideals of $R$.  The latter are precisely the ideals of $R$ that stay two-sided ideals under the localization at $S$. \cite[Example 10M]{GW} is an example  of  a prime ring  $S^{-1}R$ such that the ideal $\s_S^{-1}(0)$  of $R$ is not a prime or even  semiprime ideal. \\

{\bf Criterion for the left ideal $S^{-1}\ga$ being an ideal of $S^{-1}R$ where $\ga $ is an ideal of $R$.}
 Let $R$ be a  ring, $S\in \Den_l(R,\ga)$, $\bR :=R/\ga$ and $\pi_\ga : R\ra \bR$, $r\mapsto r+\ga$. Then $$\bS :=\pi_\ga (S)\in \Den_l(\bR , 0)\;\; {\rm  and}\;\;S^{-1}R\simeq \bS^{-1}\bR.$$  For a right  $R$-module $M$, the set $$\tor_{r,S}(M):=\{ m\in M\, | \, ms=0\;\; {\rm for \; some}\;\; s\in S\}$$ is the {\bf set of $S$-torsion elements} of $M$.  In general,  $\tor_{r,S}(M)$ is {\em not} a submodule of $M$ unless $S$ is a right denominator set of $R$.

Suppose  that $\gb$ is an ideal of a ring $R$. In general, the left ideal  $S^{-1}\gb$ of the ring  $S^{-1}R$   is {\em not} an ideal. Proposition \ref{A11Sep23} is a criterion for the left ideal $S^{-1}\gb$ being an ideal of $S^{-1}R$.

\begin{proposition}\label{A11Sep23}
Let $R$ be a  ring, $S\in \Den_l(R,\ga)$ and  $\gb$ be an ideal of $R$. Then the following statements are equivalent:

\begin{enumerate}

\item The left ideal $S^{-1}\gb$ of the ring $S^{-1}R$ is an ideal.

\item For all elements $s\in S$, $\gb s^{-1}\subseteq S^{-1}\gb$.

\item If $rs\in \gb$ for some elements $r\in R$ and  $s\in S$ then $s'r\in \gb$ for some $s'\in S$. 

\item $rs\in \ga +\gb$ for some elements $s\in S$ and $r\in R$ then $s'r\in \ga +\gb$ for some $s'\in S$.

\item  $\tor_{r,\bS}(\bR/\pi_\ga (\gb ) )  \subseteq \tor_{l,\bS}(\bR/\pi_\ga (\gb ) ) $.

\item For each $s\in S$, the ascending chain of left ideals of the rings $S^{-1}R$, $\gb_0'\subseteq \gb_1'\subseteq \cdots\subseteq \gb_i'\subseteq \cdots$, stabilizes where $\gb_i':=\sum_{j=0}^iS^{-1}\gb s^{-j}$ for $i\geq 0$. 

\end{enumerate}
\end{proposition}

\begin{proof} $(1\Leftrightarrow 2\Rightarrow 6)$ Straightforward.

 $(2\Rightarrow 3)$ Suppose that $rs\in \gb$ for some elements $r\in R$ and  $s\in S$. Then $$\frac{r}{1}\in \gb s^{-1}\subseteq S^{-1}\gb,$$ and so  $s'r\in \gb$ for some $s'\in S$.

$(3\Rightarrow 2)$ Let $b\in \gb$ and $s\in S$. Then $bs^{-1}=s_1^{-1}r$ for some elements $s_1\in S$ and $r\in R$ or, equivalently, 
$$ts_1b=trs$$ for some element $t\in S$. It follows from the inclusion $(tr)s=ts_1b\in \gb$ that $t_1tr\in \gb$ for some element $t_1\in S$.
Now, $bs^{-1}=(t_1ts_1)^{-1}\cdot t_1tr\in S^{-1}\gb$. 

$(3\Rightarrow 4)$ $rs\in \ga +\gb$ for some elements $s\in S$ and $r\in R$. There is an element $s_1\in S$ such that  $s_1rs\in\gb$. Now, by statement 3, there is an element $s_2\in S$ such that $s_2s_1r\in \gb\subseteq \ga+\gb$.

$(4\Rightarrow 3)$ Suppose that  $rs\in \gb\subseteq \ga+\gb$ for some elements $r\in R$ and  $s\in S$. By statement 4, there is an element $s_1\in S$ such that $ s_1r\in \ga +\gb$. Hence,  $ s_2s_1r\in  \gb$ for some $s_2\in S$.

$(4\Leftrightarrow 5)$ The equivalence  is an obvious re-writing of the inclusion  $\tor_{r,\bS}(\bR/\pi_\ga (\gb ) )  \subseteq \tor_{l,\bS}(\bR/\pi_\ga (\gb ) ) $ in terms of the ring $R$ (since $\bR/\pi_\ga (\gb )\simeq R/(\ga +\gb)$).

$(6\Rightarrow 2)$ For each $s\in S$, there is a natural number  $n\geq 0$ such that $\gb_n'=\gb_{n+1}'$. Hence, $S^{-1}\gb\supseteq \sum_{i=0}^nS^{-1}\gb s^i\subseteq S^{-1}\gb s^{-1}$, and statement 2 follows.
\end{proof}

\begin{corollary}\label{aA11Sep23}
Let $R$ be a  ring, $S\in \Den_l(R,\ga)$ and  $\gb$ be an ideal of $R$. Then the left ideal $S^{-1}\gb$ of the ring $S^{-1}R$ is not an ideal iff there is an element  $s\in S$ such that  the ascending chain of left ideals of the rings $S^{-1}R$, $\gb_0'\subset \gb_1'\subset \cdots\subset \gb_i'\subset\cdots$, is strictly increasing  where $\gb_i'=\sum_{j=0}^iS^{-1}\gb s^{-j}$ for $i\geq 0$. In this case, neither the ring $R$ nor its localization $S^{-1}R$ is a left Noetherian ring. 
\end{corollary}

\begin{proof} The corollary follows from Proposition \ref{A11Sep23}.
\end{proof}

\begin{lemma}\label{a10Sep23}
If  $R$ is a prime ring and $S\in \Den_l(R, 0)$  then the ring 
 $S^{-1}R$ is a prime ring.
\end{lemma}

\begin{proof} Since $S\in \Den_l(S, 0)$, the module  ${}_RR$ is an essential left $R$-submodule of $S^{-1}R$. 
 Suppose that $\ga\gb =0$ for some ideals $\ga $ and $\gb$ of $S^{-1}R$. Then $\ga'\gb' =0$  where $\ga' =\ga\cap R$ and $\gb'=\gb\cap R$ are ideals of $R$. The ring $R$ is a prime ring. Therefore, at least one of the ideals $\ga'$ or $\gb'$ is equal to zero, say $\ga'=0$. Then $\ga=0$, and so the ring $S^{-1}R$ is a prime ring.
\end{proof}

\begin{corollary}\label{a6Oct23}
Suppose that   $R$ is a  ring,  $S\in \Den_l(R,0)$ and $\gp \in \Spec (R)$. 
\begin{enumerate}

\item Suppose that $S^{-1}\gp \cap R=\gp$.  Then $S^{-1}\gp \in \Spec (S^{-1}R)$ iff  
 the left ideal $S^{-1}\gp$ of $S^{-1}R$ is an ideal. 
 
\item Suppose that $S^{-1}\gp \cap R\in \Spec (R)$.  Then $S^{-1}\gp \in \Spec (S^{-1}R)$ iff the left ideal $S^{-1}\gp$ of $S^{-1}R$ is an ideal.

\end{enumerate}
\end{corollary}

\begin{proof} 1. $(\Rightarrow)$ The implication is obvious.

$(\Leftarrow)$ Suppose that  the left ideal $S^{-1}\gp$ of $S^{-1}R$ is an ideal.
By the assumption, $S^{-1}\gp \cap R=\gp$. Hence,   $S^{-1}\gp\neq S^{-1}R$. Since $S\in \Den_l(S, 0)$, the module  ${}_RR$ is an essential left $R$-submodule of $S^{-1}R$. 
 Suppose that $\ga\gb \subseteq S^{-1}\gp$ for some ideals $\ga $ and $\gb$ of $S^{-1}R$. Then $$\ga'\gb' \subseteq R\cap S^{-1}\gp =\gp$$  where $\ga' =\ga\cap R$ and $\gb'=\gb\cap R$ are ideals of $R$. Since $\gp \in \Spec (R)$,    at least one of the ideals $\ga'$ or $\gb'$ is contained in $\gp$, say $\ga'\subseteq \gp$. Then $S^{-1}\gp\supseteq S^{-1}\ga'=\ga$, and so  $S^{-1}\gp\in \Spec (S^{-1}R)$.

2. By the assumption $\gp':=S^{-1}\gp \cap R\in \Spec (R)$. Then 

\begin{eqnarray*}
S^{-1}\gp'\cap R&=&S^{-1}(S^{-1}\gp \cap R)\cap R=S^{-1}\gp \cap S^{-1}R\cap R=S^{-1}\gp \cap R=\gp',\\
S^{-1}\gp'&=&S^{-1}(S^{-1}\gp \cap R)=S^{-1}\gp \cap S^{-1}R=S^{-1}\gp .\\
\end{eqnarray*}
By statement 1, $S^{-1}\gp'=S^{-1}\gp \in \Spec (S^{-1}R)$ iff the left ideal $S^{-1}\gp'=S^{-1}\gp$ of $S^{-1}R$ is an ideal.
\end{proof}

For a prime ideal $\gp$ and a left denominator set $S$, Proposition \ref{Aa6Oct23} provides  sufficient conditions for  $\s_S^{-1}(S^{-1}\gp )=\gp$. 

\begin{proposition}\label{Aa6Oct23}
Suppose that   $R$ is a  ring,  $S\in \Den_l(R,\ga)$,  $\gp \in \Spec (R)$, $S\cap \gp = \emptyset$  and the ideal $\s_S^{-1}(S^{-1}\gp )$ of $R$ is a finitely generated right $R$-module. Then  $\s_S^{-1}(S^{-1}\gp )=\gp$. 
\end{proposition}

\begin{proof} Since $\gp$ is an ideal of $R$, the left ideal of $R$, $$\gq:=\s_S^{-1}(S^{-1}\gp )=\{ r\in R\, | \, sr\in \gp\;\; {\rm  for \; some}\;\;s\in S\},$$ is an ideal. 
 Clearly, $\gp \subseteq \gq$ and for each element  $q\in \gq$ there is an element $s\in S$ such that $sq=0$. By the assumption,  the right $R$-module $\gq$ is finitely generated. Let $q_1, \ldots , q_m$ be a generating set.  Then we can fix an element $s\in S$ such that $sg_i\in \gp$ for all $i=1, \ldots , m$. Hence, $\gp\supseteq s\gq$, and so 
$$\gp\supseteq (s) \gq,$$
 and $\gq\subseteq \gp$ since $s\not\in \gp$ (since $\gp \cap S=\emptyset$). Therefore, $\gp = \gq$.
\end{proof}

Proposition \ref{Xa10Sep23} is a generalization of Lemma \ref{a10Sep23}. 

\begin{proposition}\label{Xa10Sep23}
If  $R$ is a ring and $S\in \Den_l(R, \gq)$ such that  $\gq \in \Spec (R)$. Then 
 $S^{-1}R$ is a prime ring.
\end{proposition}

\begin{proof} Suppose that $\ga\gb =0$ for some ideals $\ga $ and $\gb$ of $S^{-1}R$. Then $\ga'\gb' \subseteq\gq$  where $\ga' =\s_S^{-1}(\ga )$ and $\gb'=\s_S^{-1}(\gb )$ are ideals of $R$. Since $\gq \in \Spec (R)$, at least one of the ideals $\ga'$ or $\gb'$ is a subset of $\gq$, say $\ga'\subseteq \gq$. Then $\ga=S^{-1}\s_S^{-1}(\ga) \subseteq S^{-1}\gq=0$, and so the ring $S^{-1}R$ is a prime ring.
\end{proof}

Lemma \ref{b14Oct23}  is a criterion for an epimorphic image of a left denominator set being a left denominator set.

\begin{lemma}\label{b14Oct23}
Suppose that   $R$ is a  ring,  $S\in \Den_l(R,\ga )$ and $\gb$ be an ideal of $R$ such that $\ga \subseteq \gb$. Let $\pi_\gb : R\ra R/\gb$, $r\mapsto \br = r+\gb$ and $\bS=\pi_\gb (S)$. 
 Then $\bS\in \Den_l(R/\gb ,0)$ iff $\bS\subseteq\CC_{R/\gb}$.
\end{lemma}

\begin{proof} $(\Rightarrow)$ The implication is obvious.

$(\Leftarrow)$ The inclusion $\bS\subseteq\CC_{R/\gb}$ implies that  $\bS\in \Ore_l(R/\gb)$ is a left Ore set of the ring $\gb$ that consists of regular elements of $\bR$, i.e. $\bS\in \Den_l(R/\gb ,0)$.
\end{proof}

Lemma \ref{c14Oct23} is  generalization of Lemma \ref{b14Oct23}.

\begin{lemma}\label{c14Oct23}
Suppose that   $R$ is a  ring,  $S\in \Den_l(R,\ga )$, $\gb$ be an ideal of $R$ such that $S\cap (\ga + \gb)=\emptyset$. Let $\pi_{\ga+\gb} : R\ra \bR:=R/(\ga+\gb)$, $r\mapsto \br = r+\ga+\gb$,  $\bS=\pi_{\ga+\gb} (S)$, and  $\overline{\gc}:=\ass_l(\bS)$. Let 
$\pi_{\overline{\gc}} : \bR\ra \bR/\overline{\gc}$, $r\mapsto \tilde{r} = \br+\overline{\gc}$ and $\tilde{S}=\pi_{\overline{\gc}} (\bS)$.  Then $\tilde{S}\in \Den_l(\bR/\overline{\gc} ,0)$ iff $\tor_{r,S}(\bR/\overline{\gc})=0$ (i.e. if $\br s\in \overline{\gc}$ for some $\br\in R$ and $s\in S$ then $\br\in\overline{\gc}$). 
\end{lemma}

\begin{proof} The condition that $S\cap (\ga + \gb)=\emptyset$ implies that $\bS\in \Ore_l(\bR, \overline{\gc})$. Then $\tilde{S}\in \Ore_l(\bR/\overline{\gc}, 0)$. Now, the lemma is obvious.
\end{proof}

{\bf The set of minimal primes of the ring $S^{-1}R$  where $S\in \Den_l(R)$.} For an ideal $\ga$ of  a ring $R$, let $\min (\ga )$ be the set of  minimal prime ideals over $\ga$. So, the map 
\begin{equation}\label{minga}
\min (\ga)\ra \min (R/ \ga), \;\; \gp \mapsto \gp/\ga
\end{equation}
is a bijection where $\min (R/\ga)$ is the set of minimal  prime ideals of the factor ring $R/\ga$. For a ring $R$,  the ideal  $$\gn_R:=\bigcap_{\gp\in \Spec (R)}\gp=\bigcap_{\gp\in \min (R)}\gp$$  is called the  {\bf prime radical} of $R$.  If $\gn_R=0$, the ring $R$ is called a {\bf semiprime ring}. The ring $R$ is a semiprime ring iff the zero ideal is the only nilpotent ideal of $R$. 
 Then  the map 
\begin{equation}\label{minga1}
\min (R)\ra \min (R/ \gn_R), \;\; \gp \mapsto \gp/\gn_R
\end{equation}
is a bijection. So, in dealing with minimal primes of a ring without less of generality we may assume that the ring is a {\em semiprime}  ring.  For $S\in \Den_l(R)$, let 
$$\min (R, S):=\{ \gp\in \min (R)\, | \, \gp\cap S=\emptyset\}\;\; {\rm  and} \;\;\min (R, S, \id):=\{ \gp\in \min (R)\, | \, S^{-1}\gp S^{-1}R\neq S^{-1}R\}.$$ 
 Since $S^{-1}\gp \supseteq S^{-1}\gp S^{-1}R$,  $$\min (R, S)\subseteq \min (R, S, \id).$$
 
 \begin{definition}
A ring is called a {\bf prime rich ring} if each  ideal contains a finite product of prime ideals that contain it. 
\end{definition}

\begin{example} 
Every one-sided Noetherian ring $R$ is prime rich (since $|\min (R)|<\infty$ and the prime radical $\gn_R$ is a nilpotent ideal). 
\end{example}

\begin{proposition}\label{A29Sep23}
\begin{enumerate}

\item Suppose that  $R$ is  a  ring  such that $\gp_1\cdots \gp_n=0$ for some prime ideals $\gp_i$ of $R$. Then $\min (R)\subseteq \min \{ \gp_1, \ldots , \gp_n\}$ and $|\min (R)|\leq n<\infty$.

\item Every ideal of a prime rich ring has only finitely many minimal primes.
\end{enumerate}

\end{proposition}

\begin{proof} 1. Let $\gq\in \min (R)$. Then the inclusion $\gp_1\cdots \gp_n=\{ 0\}\subseteq \gq$ implies the inclusion $\gp_i\subseteq \gq$ for some $i$, and so $\gp_i=\gq$, by the minimality of $\gq$, and the result follows.

2. Statement 2 follows from statement 1. In more detail, if $\ga$ is an ideal of $R$ then $\gq_1\cdots \gq_m\subseteq \ga$ for some prime ideals $\gq_i$ of $R$ that contains $\ga$.   Let $R\ra \bR=R/\ga$, $r\mapsto \br =r+\ga$. Then $$\overline{\gq}_1\ldots \overline{\gq}_m=0\;\; {\rm  and }\;\;\{ \overline{\gq}_1,\ldots ,\overline{\gq}_m\}\subseteq \Spec (\bR).$$ By statement 1, $\min (\overline{\ga}) = \min \{\overline{\gq}_1,\ldots ,\overline{\gq}_m \} $, and so
$\min (\ga) = \min \{\gq_1,\ldots ,\gq_m \} $ and statement 2 follows.
\end{proof}

\begin{proof}{\bf(Proof of Proposition \ref{aA29Sep23})}    $(1\Rightarrow 3)$  Suppose that the ring $R$ is a prime rich ring. Then each  ideal $\ga$ of $R$ contains a product of prime ideals,  say $\gp_1\cdots \gp_n$. By  Proposition \ref{A29Sep23}.(1),
 $$\min (\ga)\subseteq \{\gp_1, \ldots , \gp_n\}.$$ 
Recall that the map $\min (\ga)\ra \min (R/\ga)$, $\gp \mapsto\gp/\ga$ is a bijection. 
Now, it follows from the inclusions 
$$ \ga \supseteq \gp_1\cdots \gp_n\supseteq \bigg(\bigcap_{\gp\in \min (\ga)}\gp\bigg)^n$$
that  the prime radical $\gn_{R/\ga}$ of the factor ring $R/\ga$ is a nilpotent ideal.

$(3\Rightarrow 2)$ Let $\ga$ be an ideal of $R$. 
 By the assumption, $|\min (\ga)|<\infty$ and the prime radical $\gn_{R/\ga}$ of the factor ring $R/\ga$ is a nilpotent ideal. Hence,
 $$\ga\supseteq \bigg(\bigcap_{\gp\in \min (\ga)}\gp\bigg)^n\supseteq  \bigg(\prod_{\gp\in \min (\ga)}\gp\bigg)^n,$$
 as required. 
 
$(2\Rightarrow 1)$ Straightforward.
\end{proof}

\begin{definition} A left denominator set $S$ of a ring $R$ is called {\bf an ideal preserving} left denominator set (or a left denominator set that  {\bf respects the ideal structure} of $R$)  if the map $S^{-1}: \CI (R)\ra \CI(S^{-1}R), \;\; I\mapsto S^{-1}I$ is a well-defined map, i.e.  for each  ideal $I$ of $R$, $S^{-1}I$ is an ideal of the ring $S^{-1}R$.  The set of all ideal preserving left denominator sets is denoted by $\Denlip (R)$. A ring $R$ is called an  {\bf ipl-ring}  (an {\bf ideal preserving ring under left localizations}) if $\Denlip (R) = \Den_l(R)$, i.e.  every left denominator set of $R$ is ideal preserving.
\end{definition}

\begin{definition} We say that a left denominator set $S$ of a ring $R$ {\bf respects the ideal structure of primes} of $R$ if for each prime ideal $\gp\in \Spec (R)$ the left ideal  $S^{-1}\gp$ of the ring $S^{-1}R$ is an ideal. 
\end{definition} 
If the left denominator set respects the ideal structure of a ring then it also the respects the ideal structure of primes. By Corollary \ref{a6Oct23},   if a left denominator set $S$ of a ring $R$ respects the ideal structure of primes then it also respects primeness of those prime ideals of $R$ that satisfy the assumption  of Corollary \ref{a6Oct23}.

\begin{example} 
Every denominator set that belong to the centre of a ring respects the ideal structure. All denominator sets of a commutative ring respect the ideal structure of the ring.
\end{example}

\begin{proposition}\label{B29Sep23}
Let $R$ be a ring and $S\in \Den_l(S,\ga)$. Suppose that the ring $S^{-1}R$ is a left Noetherian ring (e.g. $R$ is a left Noetherian ring) then $S$ respects the ideal structure of $R$. 
\end{proposition}

\begin{proof} Let $\gb$ be an ideal of the ring $R$ and $ s\in S$. For each natural number  for $i\geq 0$, let $\gb_i=\sum_{j=0}^iS^{-1}\gb s^{-j}$. By the assumption, the ring $S^{-1}R$ is a left Noetherian ring. So, the ascending chain of left ideals of the ring $S^{-1}R$,
$$ \gb =\gb_0\subseteq \gb_1\subseteq \cdots$$
stabilizes, say on $n$'th step. Then 
$S^{-1}\gb s^{-n-1}\subseteq \gb_n$, and so 
$$ S^{-1}\gb s^{-1}\subseteq \gb_ns^n=\sum_{j=0}^nS^{-1}\gb s^j\subseteq S^{-1}\gb,$$
i.e. the left ideal $S^{-1}\gb$ of the ring $S^{-1}R$ is an ideal. 
\end{proof}

 





\begin{proof} {\bf (Proof of Proposition \ref{29Sep23})} 1. (i) $|\min (R|<\infty$:  The ring $R$ is prime rich. By Proposition \ref{A29Sep23}.(2), $|\min (R|<\infty$.

(ii) $\min (S^{-1}R)\subseteq \min \{ S^{-1}\gp \, | \, \gp \in \min (R,S)\}$:  Let $\gq\in \min (S^{-1}R)$.  Then the ideal $\s_S^{-1}(\gq)$ of the prime rich ring $R$ 
  contains a finite product of prime ideals of  $R$, say $\gp_1\cdots \gp_n$, that contain $\s_S^{-1}(\gq)$.  The left denominator set  $S\in \Den (R, \ga)$ respects the ideal structure of primes. So,
$$S^{-1}\gp_1\cdots S^{-1}\gp_n=S^{-1}\gp_1\gp_2\cdots \gp_n\subseteq S^{-1}\s_S^{-1}(\gq)=\gq.$$
Hence, $S^{-1}\gp_i\subseteq \gq$ for some $i$ since  $\gq\in \Spec (S^{-1}R)$, and so $ \gp_i\subseteq \s_S^{-1}(S^{-1}\gp_i)\subseteq \s_S^{-1}(\gq)$. There exists a minimal prime ideal $\gp \in \min (R)$ such that $\gp\subseteq \gp_i$. Since $S^{-1}\gp\subseteq S^{-1}\gp_i\subseteq \gq$, we have that $\gp \in \min (R,S)$. Then, by the assumption, $S^{-1}\gp\in \Spec (S^{-1}R)$.
 Hence,
 $$S^{-1}\gp\subseteq  S^{-1}\s_S^{-1}(\gq)=\gq,$$
  and so $\gq=S^{-1}\gp$, by the minimality of $\gq$,  and the statement (ii) follows.
  
(iii) $\min (S^{-1}R)=\min \{ S^{-1}\gp \, | \, \gp \in \min (R,S)\}$: The statement (iii) follows from the statement (ii). 
  
(iv) $1\leq |\min (R,S)|\leq |\min (R)|<\infty$:
 The set of minimal primes of an arbitrary ring is a non-empty set. Now, the statement (iv) follow from the statements (i) and (iii).

2. Statement 2 follows from statement 1.

3. (i) {\em For each} $\gp \in \min (R,S)$, $\s_S^{-1}(S^{-1}\gp )=\gp$: The statement (i) follows from Proposition \ref{Aa6Oct23}


(ii) $\min (S^{-1}R)= \{ S^{-1}\gp \, | \, \gp \in \min (R,S)\}$: In view of statement 2, we have to show that the ideals $S^{-1}\gp$ and $S^{-1}\gp'$ are incomparable for all distinct prime ideals $\gp,\gp'\in \min (R,S)$. Suppose that  $ S^{-1}\gp\subset S^{-1}\gp'$,  a strict inclusion. Then, by the statement (i),  
$$ \gp= \s_S^{-1}( S^{-1}\gp)\subseteq 
\s_S^{-1}( S^{-1}\gp')= \gp'.$$
Hence, $\gp=\gp'$ (since $\gp,\gp'\in \min (R)$),   a contradiction.
\end{proof}

\begin{proof} {\bf (Proof of Corollary \ref{a29Sep23})}   (i) {\em The ring $R$ is a prime rich ring}: The prime radical of an one-sided Noetherian ring is a nilpotent ideal. For each ideal $\ga$ of $R$, the factor ring $R/\ga$ is also one-sided Noetherian. Hence, there is a natural number $n$ such that  
$$\ga\supseteq \bigg(\bigcap_{\gp \in \min (\ga)}\gp \bigg)^n\supseteq \bigg(\prod_{\gp \in \min (\ga)}\gp \bigg)^n.$$

(ii) {\em $S$ respects the ideal structure of the ring $R$}: Since the ring $S^{-1}R$ is a left Noetherian, the statement (ii) follows from Proposition \ref{B29Sep23}.

(iii) {\em For each} $\gp \in \min (R,S)$, $\s_S^{-1}(S^{-1}\gp )=\gp$: Since the ring $R$ is a right Noetherian ring, the statement (iii) follows from Proposition \ref{Aa6Oct23}.

(iv) {\em $S^{-1}\gp\in \Spec (S^{-1}R)$ for all} $\gp\in \min (R, S)$: The statement (iv), follows from Corollary \ref{a6Oct23}.(1) and 
 the statements (ii) and  (iii). 

Now, the corollary follows from Proposition statements (i), (iii) and (iv),   and  \ref{29Sep23}.(3).
\end{proof}


\section{Description of the set of minimal primes of a localization of a semiprime ring at a regular denominator set} \label{DESCRMINPR} 

 The aim of this section is to prove Theorem \ref{A10Sep23}.(2) that provides  an explicit description of  the set  of minimal primes of localizations of  semiprime  rings provided that they  have only finitely many minimal primes. Theorem \ref{A10Sep23}.(2) is one of the  key facts in  understanding   the structure of  embeddings of localizations of semiprime  rings  into  semisimple or  semiprimary rings, see \cite{Embed-SPrime-SSA}. As a corollary, for a semiprime ring $R$ with finitely many minimal primes, Corollary \ref{c10Sep23} describes the set of minimal primes of the largest left quotient ring $Q_l(R)$ of $R$. Corollary \ref{c10Sep23}  shows that the largest left quotient ring $Q_l(R)$ of a semiprime ring $R$ is also a semiprime ring and provides an explicit description of the set of the minimal primes $\min (Q_l(R))$ provided that $|\min (R)|<\infty$. Corollary \ref{aA10Sep23} is a generalization of Theorem \ref{A10Sep23}. Theorem \ref{28Sep23} is a generalization of Theorem \ref{A10Sep23}.(2,3) for left denominator sets that contain zero divisors.     A proof of Proposition \ref{b28Sep23} is given in this section. \\

{\bf Characterization of the set of minimal primes of a semiprime ring.}
\begin{definition} A finite set $Q$ of ideals of a ring  $R$ with zero intersection is called an {\bf irredundant  set} of ideals if $\bigcap_{\gq \in Q\backslash \{ \gq'\}}\gq\neq 0$ for all $\gq'\in Q$. 
\end{definition}

Lemma \ref{b10Sep23} is a useful characterization of  the set $\min (R)$ of minimal primes of a semirpime ring $R$  provided $|\min (R)|<\infty$. 

\begin{lemma}\label{b10Sep23}
Let $R$ be a semiprime ring with $|\min (R)|<\infty$. Then the  set $\min (R)$ is the only irredundant set that consist of prime ideals of $R$.
\end{lemma}
\begin{proof} Evidently, the  set $\min (R)$ is an  irredundant set that consist of prime ideals of $R$.

Suppose that $Q$ is an  irredundant set that consist of prime ideals of $R$. We have to show that $Q=\min (R)$. Since the set $\min (R)$ is an irredundant set of prime ideals, it suffices to show that $Q\supseteq\min (R)$. For each $\gp \in  \min (R)$,
$$ \bigcap_{\gq\in Q}\gq =\{0\}\subseteq \gp,$$
and so $\gq\subseteq \gp$ for some $\gq \in Q$. Hence $\gp=\gp$, by the minimality of $\gp$. Therefore, $Q\supseteq\min (R)$. 
\end{proof}

{\bf Description of the set of minimal primes of a localization of a semiprime ring at a regular denominator set.} 
 If $T$ is an ideal of $R$ then  non-empty subset $T$ of a ring $R$, the sets $\lann_R(T)$ and $\rann_R(T)$ are also ideals. Ideals of that kind are called  {\bf annihilator ideals}.

A submodule of a module is called an {\bf essential submodule} if it meets all the nonzero submodules of the module. 
In  a semiprime ring   the left annihilator of an ideal  is equal to  its  right annihilator and vice versa.

\begin{proof} {\bf (Proof of Theorem \ref{A10Sep23} )} Since $S\in \Den_l(S, 0)$, ${}_RR$ is an essential left $R$-submodule of $S^{-1}R$. 

1. Suppose that $\ga$ is a nonzero  nilpotent ideal of $S^{-1}R$ then $\ga' =R\cap \ga$ is a  nonzero  nilpotent ideal of $R$, a contradiction (since $R$ is a semiprime ring). Therefore, $R$ is a semiprime ring.

2 and 3. The case where $|\min (R)|=1$, i.e. the ring $R$ is a prime ring, follows from Lemma \ref{a10Sep23}.

 So, we assume that $|\min (R)|\geq 2$.
 For each $\gp \in \min (R)$, let $\gp^c:=\bigcap_{\gq\in \min (R)\backslash \{ \gp\} }\gq$. By  Lemma \ref{b10Sep23}, $\gp^c\neq 0$. 

(i) {\em For all $\gp \in \min (R)$, $\ann (\gp^c)=\gp$}: The ring $R$ is a semiprime ring, i.e. $\bigcap_{\gp \in \min (R)}\gp=0$. Hence, $\gp\gp^c=0$. Therefore, $$\gp\subseteq \ann_R(\gp^c).$$ Suppose that $\ga\gp^c=0$  for some ideal $\ga$ of $R$. Then $\ga\gp^c\subseteq \gp$, and so $$\ga\subseteq \gp$$ (since  $\gp^c\not\subseteq \gp$), and the statement (i) follows.

(ii) {\em For all $\gp \in \min (R)$, $\ann (\gp)=\gp^c$}: By the statement (i), $\ann (\gp)\supseteq \gp^c$. Suppose that $\ga\gp=0$ for some ideal $\ga$ of $R$.  Then $\ga\gp\subseteq \gq$ for all $\gq \in \min (R)\backslash \{ \gp\}$, and so $$\ga\subseteq \gq$$ (since  $\gp\not\subseteq \gq$), i.e.  $\ga\subseteq \gp^c$ and the statement (ii) follows.

(iii) {\em For all $\gp \in \min (R)$, $S\cap \gp =\emptyset$}: Since $\gp^c\neq 0$, $\gp^c\gp = \gp \gp^c=0$ (the statement (i)),  and $S\subseteq \CC_R $, we must have 
$ S\cap \gp =\emptyset $ for all $\gp \in \min (R)$.

Let $\pi_\gp : R\ra R/\gp$, $r\mapsto r+\gp$.

(iv)  {\em For all $\gp \in \min (R)$,  $\pi_\gp (S)\subseteq \CC_{R/\gp}$} and $\pi_\gp (S)\in \Den_l(R/\gp, 0)$: Suppose that $\pi_\gp (S)\not\subseteq \CC_{R/\gp}$. Then there are elements $s\in S$ and $r\in R\backslash \gp$ such that  either $sr\in\gp$ or $rs\in \gp$. Let us to consider the first case as the arguments in the second case are `symmetrical' to the ones in the first case. It follows from the inclusion $sr\in\gp$ and the statement (i) that $sr\gp^c=0$. Hence,  $r\gp^c=0$ since $s\in S\subseteq \CC_R$. By the statement (i), $r\in \gp$, a contradiction. The image of a left Ore set under a ring  epimorphism is a left Ore set provide that it does not contains $0$. Hence,  $\pi (S)\in \Den_l(R/\gp, 0)$ since $\pi_\gp (S)\subseteq \CC_{R/\gp}$.

(v) {\em For each  $\gp \in \min (R)$, the left ideal  $S^{-1}\gp$ of $S^{-1}R$ is an ideal}: By the statement (iv), $$\tor_{r,S}(R/\gp )=0.$$ Hence 
  $S^{-1}\gp$ is an ideal  of $S^{-1}R$, by Proposition \ref{A11Sep23}.(1).

 
(vi) $\{ S^{-1}R\, | \, \gp \in \min (R)\}\subseteq \Spec (S^{-1}R)$ {\em and} $S^{-1}(R/\gp)\simeq S^{-1}R/S^{-1}\gp\simeq \pi_\gp (S)^{-1}(R/\gp)$:  By Lemma \ref{a10Sep23}, for each $\gp \in \min (R)$, the ring $S^{-1}(R/\gp)$ is a prime ring. By the statements (iv) and  (v), $S^{-1}(R/\gp)\simeq S^{-1}R/S^{-1}\gp\simeq \pi_\gp (S)^{-1}(R/\gp)$,  and so 
$$S^{-1}\gp\in \Spec (S^{-1}R).$$

(vii) {\em The set $\{ S^{-1}R\, | \, \gp \in \min (R)\}$ is an irredundant set of prime ideals of $S^{-1}R$}: Since $S\in \Den_l(R,0)$, we have that   
\begin{eqnarray*}
0&=&S^{-1}0=S^{-1}\bigg(\bigcap_{\gp \in \min (R)}\gp \bigg)=\bigcap_{\gp \in \min (R)}S^{-1}\gp,  \\
 0&\neq& \bigcap_{\gp\in \min (R)\backslash \{ \gp'\} }\gp\subseteq S^{-1}\bigg(\bigcap_{\gp\in \min (R)\backslash \{ \gp'\} }\gp \bigg)=\bigcap_{\gp\in \min (R)\backslash \{ \gp'\} }S^{-1}\gp\;\; {\rm for\; all}\;\; \gp'\in \min (R),
\end{eqnarray*}
and the statement (vii) follows.

(viii) $\min (S^{-1}R)=\{ S^{-1}R\, | \, \gp \in \min (R)\}$: The statement (viii) follows from the statement (vii) and Lemma \ref{b10Sep23}.
\end{proof}

Corollary \ref{c10Sep23}.(1) shows that the largest left quotient ring $Q_l(R)$ of a semiprime ring $R$ is also a semiprime ring. Corollary \ref{c10Sep23}.(2) and provides an explicit description of the set of its minimal primes provided that $|\min (R)|<\infty$.

\begin{corollary}\label{c10Sep23}
Let $R$ be a semiprime ring. Then 
\begin{enumerate}

\item The ring $Q_l(R)$ is a semiprime ring.

\item If, in  addition,  $|\min (R)|<\infty$.  Then 
\begin{enumerate}

\item $\min (Q_l(R))=\{ S_l(R)^{-1}\gp\, | \, \gp \in \min (R)\}$.

\item For all $\gp \in \min (R)$, $\pi_\gp (S_l(R))\in \Den_l(R/\gp , 0)$ and  $$\pi_\gp (S_l(R))^{-1}(R/\gp)\simeq S_l(R)^{-1}R/S_l(R)^{-1}\gp\simeq Q_l(R)/S_l(R)^{-1}\gp \simeq S_l(R)^{-1}(R/\gp)$$ where $\pi_\gp : R\ra R/\gp$, $r\mapsto r+\gp$.

\item For all $\gp \in \min (R)$, $\pi_\gp (S_l(R))\subseteq S_l(R/\gp)$ and $Q_l(R)/S_l(R)^{-1}\gp \subseteq Q_l(R/\gp)$.

\end{enumerate}
\end{enumerate}
\end{corollary}

\begin{proof} 1 and 2. By the definition, $Q_l(R)=S_l(R)^{-1}R$ and $S_l(R)\in \Den_l(R, 0)$, and statement 1 and 2 follow from Theorem  \ref{A10Sep23}.(1,2).

3. Statement 3 is a particular case of Theorem \ref{A10Sep23}.(3). 

4. By the definition, $S_l(R/\gp)$ is the largest left denominator set of  the ring $R/\gp$ that consists of regular elements of  $R/\gp$. By statement 3, $$\pi_\gp (S_l(R))\in \Den_l(R/\gp , 0),$$  and so  $\pi_\gp (S_l(R))\subseteq S_l(R/\gp)$ and $Q_l(R)/S_l(R)^{-1}\gp \subseteq Q_l(R/\gp)$.
\end{proof}

Corollary \ref{aA10Sep23} is a generalization of Theorem \ref{A10Sep23}.

\begin{corollary}\label{aA10Sep23}
Let $R$ be a  ring, $S\in \Den_l(R, \ga)$,  $\ga$ be a semiprime ideal, $\bR=R/\ga$ and $\bS=\{ s+\ga\, | \, s\in S\}$. Then 
\begin{enumerate}

\item The ring $S^{-1}R\simeq \bS^{-1}\bR$ is a semiprime ring.

\item If, in  addition,  $|\min (\bR)|<\infty$ then $\min (S^{-1}R)=\{ S^{-1}\gp =\bS^{-1}\gp\, | \, \gp \in \min (\bR)\}$, i.e. the map $\min (\bR)\ra \min (S^{-1}R)$, $\gp \mapsto S^{-1}\gp$ is a bijection and $|\min (S^{-1}R)|=|\min (\bR)|$.

\item If, in  addition,  $|\min (\bR)|<\infty$ then
$\pi_\gp (\bS)\in \Den_l(R/\gp , 0)$ and  $\pi_\gp (\bS)^{-1}(\bR/\gp)\simeq \bS^{-1}\bR/\bS^{-1}\gp \simeq \bS^{-1}(\bR/\gp)$ where $\pi_\gp : \bR\ra \bR/\gp$, $r\mapsto r+\gp$.
\end{enumerate}
\end{corollary}

\begin{proof} The ideal $\ga$ is a semiprime ideal of $R$. Hence, the factor ring $\bR=R/\ga$ is a semiprime ring. It follows from the inclusion $S\in \Den_l(S, \ga)$ that  $\bS\in \Den_l(\bR, 0)$, $S^{-1}R\simeq \bS^{-1}\bR$ is an $R$-isomorphism. Now, the corollary follow from Theorem \ref{A10Sep23}. 
\end{proof}

\begin{proposition}\label{A15Sep23}
Let $R$ be a semiprime ring with $|\min (R)|<\infty$ and $*\in \{ l,r,\emptyset\}$. Then 
\begin{enumerate}
\item $S_*(R)\subseteq \bigcap_{\gp \in \min (R)}\pi_\gp^{-1} (S_*(R/\gp ))$.
\item The homomorphism  
$$\prod_{\gp \in \min (R)}\pi_\gp':Q_*(R)\ra \prod_{\gp \in \min (R)}Q_*(R/\gp)$$
 is an $R$-monomorphism where the homomorphisms  $\pi_\gp' :Q_*(R)\ra Q_*(R/\gp)$  are defined  in the commutative diagram below. 
\end{enumerate}
\end{proposition}

\begin{proof}  Let us prove the result for $*=l$. 
The other two cases can be treated in a similar way. 

1. By Theorem \ref{A10Sep23}.(3), $\pi_\gp (S_l(R))\in \Den_l(R/\gp , 0)$ and so $\pi_\gp (S_l(R))\subseteq S_l(R/\gp)$, by the maximality of $S_l(R/\gp)$, and there is a commutative diagram of ring homomorphisms:
 \begin{displaymath}
    \xymatrix{
        R \ar[r]^{\pi_\gp} \ar[d]_{\s} & R/\gp \ar[d]^{\s_\gp} \\
        Q_l(R) \ar[r]^{\pi_\gp'}       & Q_l(R/\gp) }
\end{displaymath}
where $\s (r)=\frac{r}{1}$, $\s_\gp (\pi_\gp (r))=\frac{\pi_\gp (r)}{1}$ and $\pi_\gp'(s^{-1}r)=\pi_\gp (s)^{-1}\pi_\gp (r)$ for all $r\in R$ and $s\in S_l(R)$. Now, statement 1 follows from Theorem \ref{4Jul10}.(1) and the commutative diagram above. 

2. By the commutative diagram above,  the map 
$$\prod_{\gp \in \min (R)}\pi_\gp':Q_*(R)\ra \prod_{\gp \in \min (R)}Q_*(R/\gp)$$
 is an $R$-monomorphism. 
\end{proof}

For an ideal $\ga$ of a ring $R$, let 
 $\Ore_l(R, \ga ) := \{ S\in \Ore_l(R)\, | \, \ass_l (S)=\ga \}$,  
 $\Den_l(R, \ga ) := \{ S\in \Den_l(R)\, | \, \ass (S)=\ga \}$. Let $\Ass_l(R):= \{ \ass_l (S)\, | \, S\in \Den_l(R)\}$.

\begin{definition}
(\cite{larglquot}) Let  $S_{l,\ga }(R)$
be  the  largest element of the poset $(\Den_l(R, \ga ),
\subseteq )$ ($S_{l,\ga }(R)$ exists,\cite[Theorem 2.1.(2)]{larglquot}). The ring  $$Q_{l,\ga }(R):=S_{l,\ga}(R)^{-1} R$$ is called  the
{\bf largest left quotient ring associated to} $\ga$.
\end{definition}

For each denominator set $S\in \Den_l(R, \ga )$ where $\ga := \ass_l
(S)$, there are natural ring homomorphisms $R\stackrel{\pi}{\ra}
R/ \ga \ra S^{-1}R$.

\begin{lemma}\label{a20Sep23}
(\cite[Lemma 3.3.(2)]{larglquot})
 Let $\ga\in \Ass_l(R)$ and
$\pi : R\ra R/ \ga$, $a\mapsto a+\ga$. Then $\pi^{-1} (S_l(R/ \ga
)) = S_{l,\ga} (R)$, $\pi (S_{l,\ga} (R)) = S_l(R/ \ga )$ and  $ Q_{l,\ga}
(R) = Q_l( R/ \ga )$. 
\end{lemma}

Let $R$ be a semiprime ring with $|\min (R)|<\infty$. For each $\gp\in \min (R)$, the subset of $R$, 
$$T_l(\gp):= \Big(\s_\gp\pi_\gp\Big)^{-1}(Q_l(R/\gp)^\times )=\pi_\gp^{-1}(S_l(R/\gp)),
$$ is a multiplicative set, see the diagram in the proof of Proposition \ref{A15Sep23}.(1) for the definitions of the maps $\pi_\gp$ and $\s_\gp$. The second equality follows from the  equality
$S_l(R/\gp)=\s_\gp^{-1}(Q_l(R/\gp)^\times)$ (Theorem \ref{4Jul10}.(1)). For each $\gp\in \min (R)$, let 
\begin{eqnarray*}
\ga_l(\gp)&:=& \ass_l(T_l(\gp))=\{ r\in R\, | \, tr=0\; {\rm  for\; some} \;\;t\in T_l(\gp )\},\\
\ga_r(\gp)&:=& \ass_r(T_l(\gp))=\{ r\in R\, | \, rt=0\; {\rm  for\; some} \;\;t\in T_l(\gp )\}.
\end{eqnarray*}
By the definition, $ \ga_l(\gp)$ and $\ga_r(\gp)$
are right and left ideals of the ring $R$, respectively, but not ideals, in general. 

\begin{proposition}\label{19Sep23}
Let $R$ be a semiprime ring with $|\min (R)|<\infty$. Then 
\begin{enumerate}

\item $\ga_l(\gp)\subseteq \gp$,  $\ga_r(\gp)\subseteq \gp$ and $\bigcap_{\gp \in \min (R)}\ga_l(\gp)=\bigcap_{\gp \in \min (R)}\ga_r(\gp)=0$.

\item  $T_l(\gp) \in \Ore_l(R)$ iff for each pair of elements $(s,p)\in T_l(\gp)\times \gp$
there is a pair of elements $(s',r')\in  T_l(\gp)\times R$ such that $s'r-r's\in \ga_l(\gp)$.

\item Suppose that $T_l(\gp)\in \Den_l(R)$.  Then 
\begin{enumerate}

\item $T_l(\gp)^{-1}\gp$ is an ideal of the ring $T_l(\gp)^{-1}R$.

\item $T_l(\gp)^{-1}R/T_l(\gp)^{-1}\gp\simeq Q_l(R/\gp)$.

\item If, in addition, $\ga_l(\gp)=\gp$ then $S_{l, \gp}(R)=T_l(\gp)$ and $Q_{l, \gp}(R)\simeq Q_l(R/\gp) \simeq T_l(\gp)^{-1}R$.

\end{enumerate}

\end{enumerate}
\end{proposition}

\begin{proof} 1. Let us prove that statement 1 holds  for $\ga_l(\gp)$. The second case can be treated in a similar way (using symmetrical arguments).
 If $r\in \ga_l(\gp)$ them $tr=0$ for some element $t\in T_l(\gp)$. By the definition, $\s_\gp \pi_\gp(T_l(\gp))\subseteq Q_l(R/\gp)^\times$.  It follows that  $0=\s_\gp \pi_\gp (tr)=\s_\gp \pi_\gp (t)\s_\gp \pi_\gp (r)$, and so $$\s_\gp \pi_\gp (r)=0,$$ i.e. $\pi_\gp (r)=0$. This means that $r\in \gp$. So, $\ga_l(\gp)\subseteq \gp$. Now,  $\bigcap_{\gp \in \min (R)}\ga_l(\gp)\subseteq \bigcap_{\gp \in \min (R)}\gp=0$, and so  $\bigcap_{\gp \in \min (R)}\ga_l(\gp)=0$.

2. $(\Rightarrow )$ Suppose that  $T_l(\gp) \in \Ore_l(R)$. Then for each pair of elements $(s,p)\in T_l(\gp)\times \gp$
there is a pair of elements $(s',r')\in  T_l(\gp)\times R$ such that $s'r-r's=0\in \ga_l(\gp)$.

$(\Leftarrow )$ We have to show that the left Ore condition holds for the multiplicative subset $T_l(\gp)$. By the assumption, for each pair of elements $(s,p)\in T_l(\gp)\times \gp$
there is a pair of elements $(s',r')\in  T_l(\gp)\times R$ such that $s'r-r's\in \ga_l(\gp)$. By the definition of the set $|ga_l(\gp)$, there is an element $t\in T_l(\gp)$ such that $$t(s'r-r's)=0.$$ Then $(ts')r=(tr')s$ with $ts'\in T_l(\gp)$ and $tr'\in R$, as required.

3. By the assumption, $T_l(\gp)\in \Den_l(R)$, and so $T_l(\gp)\in \Den_l(R, \ga_l(\gp))$. Recall that $$\s_\gp \pi_\gp(T_l(\gp))\subseteq Q_l(R/\gp)^\times.$$ By the universal property of localization, there is a homomorphism
  $$\iota_\gp: T_l(\gp)^{-1}R\ra Q_l(R/\gp), \;\; t^{-1}r\mapsto t^{-1}r.$$ 
Recall that $S_l(R/\gp)=\s_\gp^{-1}(Q_l(R/\gp)^\times)$ (Theorem \ref{4Jul10}.(1)) and $T_l(\gp)=\pi_\gp^{-1}(S_l(R/\gp))$ (see the definition of the set $T_l(\gp)$). The map $\pi_\gp: R\ra R/\gp$, $r\mapsto r+\gp$ is an epimorphism. Hence,
$$ \pi_\gp (T_l(\gp))=S_l(\gp),$$
and so the homomorphism $\iota_\gp$ is an epimorphism. By applying the localization functor $T_l(\gp)^{-1}$ (which is an exact functor) to the short exact sequence of left $R$-modules
$$ 0\ra \gp \ra R\ra R/\gp\ra 0,$$
we obtain the short exact sequences of left $T_l(\gp)^{-1}R$-modules
$$ 0\ra T_l(\gp)^{-1}\gp \ra T_l(\gp)^{-1}R \ra T_l(\gp)^{-1}(R/\gp)\simeq \Big(\pi_\gp(T_l(\gp)) \Big)^{-1}(R/\gp)=S_l(R/\gp)^{-1}(R/\gp)=Q_l(R/\gp)\ra 0,$$
i.e. $ 0\ra T_l(\gp)^{-1}\gp \ra T_l(\gp)^{-1}R \stackrel{\iota_\gp}{\ra} Q_l(R/\gp)\ra 0$. Hence, $\ker (\iota_\gp)= T_l(\gp)^{-1}\gp$ is an ideal of the ring $ T_l(\gp)^{-1}R $ such that 
$$ T_l(\gp)^{-1}R/T_l(\gp)^{-1}\gp\simeq Q_l(R/\gp).$$
The proofs of the statements (a) and (b) are complete. 

Finally, suppose that $\ga_l(\gp )=\gp$, i.e. $T_l(\gp)\in \Den_l(R, \gp)$. Then, by the statement (b), 
$$\s : R\ra T_l(\gp)^{-1}R\simeq Q_l(R/\gp), \;\; r\mapsto \frac{r}{1}.$$
By \cite[Lemma 3.3.(2)]{larglquot}, 
$$ Q_{l,\ga}(R)=S_{l,\ga}(R)^{-1}R=Q_l(R/\gp )\;\; {\rm and}\;\; S_{l,\ga}(R)=\pi_\gp^{-1}(S_l(R/\gp)).$$
By the definition of the set  $T_l(\gp)$, $\pi_\gp^{-1}(S_l(R/\gp))=T_l(\gp)$. The proof of the statement (c) is complete.
\end{proof}

{\bf Description of the set of minimal primes of a localization of a semiprime ring at a not necessarily regular denominator set.} An element $n$ of a ring $R$ is called a {\bf normal element} if $$Rn=nR.$$  
  Theorem \ref{28Sep23} is a generalization of Theorem \ref{A10Sep23}.(2,3) for left denominator sets that contain zero divisors. The general situation is described by  Proposition \ref{29Sep23} but the goal is to describe  a situation that has the `same' answer as in Theorem \ref{A10Sep23}.(2,3).

\begin{theorem}\label{28Sep23}
Let $R$ be a semiprime ring with $|\min (R)|<\infty$ and $S\in \Den_l(R, \ga)$. Then 
 $$\min (R, S):=\{ \gp\in \min (R)\, | \, \gp\cap S=\emptyset\}\neq \emptyset
 $$
 and statements 1 and 2 below are equivalent: 
\begin{enumerate}

\item The ring $S^{-1}R$ is a semiprime ring with $\min (S^{-1}R)=\{ S^{-1}\gp\, | \, \gp \in \min (R,S)\}$.

\item For each $\gp \in \min (R,S)$, $ S^{-1}\gp$ is an ideal of $S^{-1}R$ and the factor ring $S^{-1}R/S^{-1}\gp$
 is a prime ring.

\item  Suppose that  the  equivalent conditions 1 and 2  hold. Then 
\begin{enumerate}
\item If the left denominator set $S$ is generated by  normal elements of $R$ then 
  all ideals in the set $\min (S^{-1}R)$ are distinct, i.e. $|\min (S^{-1}R)|=|\min (R,S)|$.
  
\item If for each element $s\in S$ there is an element $t\in S$ such that the element $ts$ is a normal element of $R$ then  all ideals in the set $\min (S^{-1}R)$ are distinct, i.e. $|\min (S^{-1}R)|=|\min (R,S)|$.
\end{enumerate}
\end{enumerate}

\end{theorem} 

{\em Remark}. In statement 1, not all ideals in the set $\min (S^{-1}R)$ are assumed distinct.

\begin{proof} Suppose that $\min (R, S)=\emptyset$, i.e. $\gp\cap S\neq \emptyset$ for all $\gp\in \min (R)$. Choose an element $s_\gp \in \gp\cap S\neq \emptyset$ for each $\gp\in \min (R)$. Then on the one hand $$0\neq s:=\prod_{\gp \in \min (R)}s_\gp\in S$$ where the product is taken in an arbitrary order. 
On the other hand,
$s\in \bigcap_{\gp \in \min (R)} \gp =0$ (since the ring $R$ is a semiprime ring), a contradiction. Therefore, $\min (R, S)\neq \emptyset$.

$(1\Rightarrow 2)$ Straightforward.

$(2\Rightarrow 1)$ By the assumption, for each $\gp \in \min (R,S)$, $ S^{-1}\gp$ is an ideal of $S^{-1}R$ and the factor ring $S^{-1}R/S^{-1}\gp$
 is a prime ring, i.e.  the ideal $ S^{-1}\gp$ is a  prime ideals of the ring $S^{-1}R$. Since the ring $R$ is a semiprime ring with $|\min (R)|<\infty$, we have that 
 $$\{0\} =S^{-1}\{0\}=S^{-1}\bigcap_{\gp \in \min (R)}\gp=\bigcap_{\gp \in \min (R)}S^{-1}\gp=
\bigcap_{\gp \in \min (R,S)}S^{-1}\gp=\bigcap_{\gp \in P }S^{-1}\gp$$
for some subset $P$ of $\min (R,S)$ such that the last  intersection is irredundant. By Lemma \ref{b10Sep23}, $\min (S^{-1}R)=\{ S^{-1}\gp \, | \, \gp \in P\}$.

3. Suppose that  the  equivalent conditions 1 and 2 of the theorem  hold.

(a) The statement (a) is a particular case of the statement (b). 

(b)    Suppose that $S^{-1}\gp=S^{-1}\gq$ for some ideals $\gp,\gq\in \min (R,S)$. We have to show that $\gp = \gq$. It suffices to show that $$\gp \subseteq \gq$$ (since then the inclusions $\gp, \gq \in \min (R)$  imply the equality $\gp = \gq$). For each element $p\in \gp $, $\frac{p}{1}=s^{-1}q$ for some element $q\in \gq$. Hence, $t(sp-q)=0$ for some element and $t\in S$, and so  $$ts\cdot p=tq\in \gq\;\; {\rm  and}\;\;ts\in S\backslash \gq$$ (since $S\cap \gq =\emptyset$).
By the assumption, there is an element $t'\in S$ such that the element $n:=t'ts\in S$ is a normal element of $R$. Now, the inclusion $np\in t'\gq\subseteq \gq$ implies the inclusion of ideals
$$(n) (p)\subseteq \gq$$ (since the element $n\in R$  is a normal element). Therefore, $(p)\subseteq \gq$ since $n\not\in \gp$, i.e.     
 $p\in \gq$ for all elements $p\in \gp$. This implies the inclusion  $\gp \subseteq \gq$, as required.  
\end{proof}

Corollary \ref{a28Sep23} describes the minimal primes of a localization of a semiprime commutative ring.

\begin{corollary}\label{a28Sep23}
Let $R$ be a semiprime commutative ring with $|\min (R)|<\infty$ and $S\in \Den (R, \ga)$. Then 
\begin{enumerate}

\item  The ring $S^{-1}R$ is a semiprime ring with $\min (S^{-1}R)=\{ S^{-1}\gp\, | \, \gp \in \min (R,S)\}$.

\item In statement 1,  all ideals in the set $\min (S^{-1}R)$ are distinct, i.e. $|\min (S^{-1}R)|=|\min (R,S)|$.
\end{enumerate}

\end{corollary}

\begin{proof} 1. Since the ring $R$ is a commutative ring, for each $\gp \in \min (R,S)$, $ S^{-1}\gp$ is an ideal of $S^{-1}R$ and the factor ring $$S^{-1}R/S^{-1}\gp\simeq S^{-1}(R/\gp)$$
 is a prime ring since the factor ring $R/\gp$ is a domain (hence so is its localization $S^{-1}(R/\gp)$). Therefore, statement 2 of Theorem \ref{28Sep23} holds, and so statement 1 follows from Theorem \ref{28Sep23}.
 
 2.  Statement 2 follows from Theorem \ref{28Sep23}.(3a).
\end{proof}

 An ideal $\gp$ of a ring $R$ is called a {\bf completely prime ideal} of $R$ if the factor ring $R/\gp$ is a domain. Every completely prime ideal is a prime ideal  but not vice versa, in general. We denote by $\Spec_c(R)$ be the set of all completely prime ideals of $R$. Notice that $\Spec_c(R)\subseteq \Spec (R)$.
 

\begin{proof} {\bf (Proof of Proposition \ref{b28Sep23})}   1.    By the assumption, for each $\gp \in \min (R,S)$,  $\gp \in \Spec_c(R)$ and 
$ S^{-1}\gp$ is an ideal of $S^{-1}R$. So,  the ring $R/\gp$ is a domains,  and  so is  the factor ring 
$$S^{-1}R/S^{-1}\gp\simeq S^{-1}(R/\gp).$$
 Therefore, statement 2 of Theorem \ref{28Sep23} holds, and so statement 1 follows from Theorem \ref{28Sep23}. Clearly, $S^{-1}\gp\in \Spec_c(S^{-1}R)$ (since the factor ring $S^{-1}R/S^{-1}\gp\simeq S^{-1}(R/\gp)$ is a domain). 
 
 2.  Suppose that $S^{-1}\gp=S^{-1}\gq$ for some ideals $\gp,\gq\in \min (R,S)$. We have to show that $\gp = \gq$. It suffices to show that $\gp \subseteq \gq$. For each element $p\in \gp $, $\frac{p}{1}=s^{-1}q$ for some element $q\in \gq$. Hence, $t(sp-q)=0$ for some element and $t\in S$, and so  $$ts\cdot p=tq\in \gq\;\; {\rm  and}\;\;ts\in S\backslash \gq$$ (since $S\cap \gq =\emptyset$). Therefore, $p\in \gq$ for all elements $p\in \gp$,    i.e. $\gp \subseteq \gq$ (since the factor ring $R/\gq$ is a domain), as required. 
\end{proof}


\section{ Minimal primes of localizations of rings at multiplicative sets generated by normal elements} \label{MINPTNORM} 

Let $R$ be a ring and $S$ be a multiplicative subset of $R$ which is generated by normal elements of $R$. The goal of the section is to prove Theorem \ref{A2Oct23} that describes the set of minimal primes of the ring $\RSm$. In general, the ring $\RSm$ is not a semiprime ring but precisely  the same result as Theorem \ref{A10Sep23} holds for it but without the `semiprimeness' assumption, see Theorem \ref{A2Oct23}.(4a). Lemma \ref{a5Oct23} is a generalization  of Theorem \ref{A2Oct23} to a more general situation. \\

{\bf Description of minimal primes of localizations of rings at multiplicative sets generated by normal elements.} Let $R$ be a ring,   $S$ be a non-empty subset of $R$,  $R\langle X_S\rangle$ be a ring freely generated by the ring $R$ and a set $X_S=\{ x_s\, | \, s\in S\}$ of free noncommutative indeterminates (indexed by the elements of the set $S$). Let $I_S$ be the  ideal of $R\langle X_S\rangle$  generated by the set  $\{ sx_s-1, x_ss-1 \, | \, s\in S\}$. The factor  ring  
\begin{equation}\label{IRbSbm}
\RSm := R\langle X_S\rangle/ I_S.
\end{equation}
 is called the {\em localization of $R$ at} $S$. 
Let $\ass (S) = \ass_R(S)$ be the  kernel of the ring homomorphism
\begin{equation}\label{IRbSbm1}
\s_S: R\ra \RSm , \;\; r\mapsto r+ I_S.
\end{equation}
  The factor ring $\bR :=R/\ass_R(S)$ is a subring of $\RSm$ and the map $\pi_S:R\ra \bR$, $r\mapsto \br := r+\ass_R(S)$ is   an epimorphism.  The  ideal $\ass_R(S)$ of $R$ has a sophisticated  structure, its description is given in \cite[Proposition 2.12]{LocSets} when  $$\RSm =\{ \bs^{-1}\br\, | \, s\in S, r\in R\}$$ is a ring of left fractions.   
 There is an example of a domain $R$ and a finite set $S$ such that $\ass (S)=R$, i.e. $\RSm =\{ 0\}$  \cite[Exercises 9.5]{Lam-Exbook}.

\begin{definition} (\cite{LocSets}) A multiplicative set $S$ of a ring $R$ is called a {\bf left localizable set}  if  
$$\RSm = \{ \bs^{-1} \br \, | \, \bs \in \bS, \br \in \bR\}\neq \{ 0\}$$
where $\bR = R/ \ga$,  $\ga = \ass_R(S)$ and $\bS = (S+\ga ) / \ga$, i.e., every element of the ring $\RSm$ is a left fraction $\bs^{-1} \br$ for some elements  $\bs \in \bS$ and $ \br \in \bR$. Similarly,  a multiplicative set $S$ of a ring $R$ is called a {\bf right  localizable set}  of $R$  if 
$$\RSm = \{  \br \bs^{-1}\, | \, \bs \in \bS, \br \in \bR\}\neq \{ 0\}, $$
 i.e., every element of the ring $\RSm$ is a right  fraction $ \br\bs^{-1}$ for some elements  $\bs \in \bS$ and $ \br \in \bR$. A right and left localizable set of $R$ is called a {\bf localizable set} of $R$.
 \end{definition}
 Every left denominator set $S\in \Den_l(R)$ is a left localizable set and $\RSm \simeq S^{-1}R$.

Let $R$ be a ring and $S$ be a multiplicative subset of $R$ which is generated by normal elements of $R$. Clearly,  $S\in \Ore (R)$ is an  Ore set of $R$. By Theorem \ref{10Jan19-st1,2}, $S$ is a localizable set of $R$ such that  $\ass_R(S)=\ga $ and $\RSm \simeq \bS^{-1}\bR$. Wen keep the notation of Theorem \ref{10Jan19-st1,2}. Let  $$\widetilde{\pi}: \bR\ra \widetilde{R}:=\bR/  \gn_{\bR}, \;\;\br\mapsto \widetilde{r}=\br +\gn_{\bR}$$ where $\gn_{\bR}$ is the prime radical of the ring $\bR$. There are bijections
$$\min (\ga)\stackrel{\pi}{\ra} \min (\bR)\stackrel{\widetilde{\pi}}{\ra}\min(\widetilde{R}), \;\;  \gp\mapsto \bgp \mapsto \widetilde{\overline{\gp}}.$$
Let $\min (\ga , S):=\{ \gp \in \min (\ga)\, | \, \gp\cap S =\emptyset\}$.

An element $a\in R$ is called {\bf strongly nilpotent} if any sequence $a=a_0,a_1, a_2, \ldots$ such that $a_{i+1}\in a_iRa_i$ is ultimately zero. The prime radical $\gn_R$ of $R$ is precisely the set of strongly nilpotent elements of $R$, \cite[Theorem 0.2.2]{MR}. In particular, the prime radical $\gn_R$ is {\bf nil}, i.e. all elements of $\gn_R$ are nilpotent elements.

\begin{proof} {\bf (Proof of Theorem \ref{A2Oct23})}   1--3.  By Theorem \ref{10Jan19-st1,2}.(2), $\bS\in \Den (\bR,0)$ and $\RSm\simeq \bS^{-1}\bR$. Hence, for each element $\bs\in \bS$, the map 
$$\o_{\bs}: \bR\ra \bR, \;\; \br\mapsto \o_{\bs}(\br)\;\; {\rm where}\;\; \bs\br =\o_{\bs}(\br ) \bs,$$
is an automorphism. 

(i) {\em For all $\bs\in \bS$, $\o_{\bs}(\gn_{\bR})=\gn_{\bR}$}: The statement (i) follows from the fact that $\o_{\bs}\in \Aut (\bR)$.

(ii) $\bS\cap \gn_{\bS}=\emptyset$: Each element  $\bs\in \bS$ is a regular element of $R$. In particular, it is not a nilpotent element. Since the prime radical $\gn_{\bS}$ is nil, we have that $\bS\cap \gn_{\bS}=\emptyset$.

(iii) {\em $\tbS\in \Den (\widetilde{R},0)$ and $\o_\xi\in \Aut (\widetilde{R})$ for all} $\xi \in \tbS$: By the statement (ii), the subset $\tbS$ of $\widetilde{R}$  is a multiplicative set. By the statement (i),  the set $\tbS$   is generated by regular normal elements. 
 Hence,  $\o_\xi\in \Aut (\tbR )$ for all $\xi \in \tbS$ and  $\tbS\in \Den (\tbR,0)$.

(iv) {\em The map $\min (\widetilde{R})\ra \min (\tbS^{-1}\widetilde{R})$, $\tbgp\mapsto \tbS^{-1}\tbgp$ is a bijection}: By the definition, the ring $\widetilde{R}$ is a semiprime ring with $$|\min (\widetilde{R})|=|\min (\ga)|<\infty, \;\;  {\rm and}\;\;  \tbS\in \Den (\widetilde{R},0)$$
 (the statement (iii)). Now, the result follows from Theorem \ref{A10Sep23}.(2).

(v) {\em For all  $\bgp\in \min (\bR)$, $\bS\cap \bgp = \emptyset$}: The result follows from the statement (iv).

(vi) {\em The map $\min (\bR)\ra \{ \bS^{-1}\bgp\, | \, \gp\in \min (\bR)\}$, $\bgp\mapsto \bS^{-1}\bgp$ is a bijection, $\bS^{-1}\bgp\neq \bS^{-1}R$ and $\bS\cap \bgp=\emptyset$}: The result follows from the statement (iv). 

Notice that $|\min (\bR )|=|\min (a)|<\infty$.

(vii) {\em There is a natural number $n\geq 1$ such that $\o_{\bs}^n(\bgp)= \bgp$ for all elements  $\bs\in \bS$ and $\bgp\in \min (\bR)$}: Let $G$ be the subgroup of $\Aut (R)$ which is generated by the automorphisms $\o_{\bs}$, where $\bs \in \bS$. The group $G$ acts on the finite set $\min (\bR)$ in the obvious way. The image of the group $G$, say $G'$, in the symmetric group of the finite set $\min (\bR)$ is a finite group. So, it suffices to take $n=|G'|$, the order of the group $G'$.

(viii) {\em For all $\bgp\in \min (\bR)$ and $\bs\in \bS$, $ \bgp\bs^{-1} \subseteq \bs^{-n}\bgp$, where $n$ is as in the statement (vii), and  the left ideal $\bS^{-1}\bgp$ of the ring $\bS^{-1}\bR$ is an ideal}: The second part of the statement (viii) follows from the first (since then $\gp\bs^{-1}\subseteq \bS^{-1}\gp$ for all element $\bs \in \bS$). Now, 
$$ \bgp\bs^{-1}=\bs^{-n}\cdot \bs^n\bgp\bs^{-n}\cdot \bs^n\bs^{-1}=\bs^{-n}\o_{\bs}^n(\bgp) \bs^{n-1}=\bs^{-n}\bgp \bs^{n-1}\subseteq \bs^{-n}\bgp .$$
(ix) {\em For all $\bgp\in \min (\bR)$,  $\bS^{-1}\bgp\in \Spec (\bS^{-1}\bR$)}:  By statement (viii), $\bS^{-1}\bgp$ is an ideal of the ring $\bS^{-1}\bR$. Furthermore,  $\bS^{-1}\bgp\neq \bS^{-1}\bR$, by the statement (vi). The ring $\bR/\bgp$ is a prime ring. Hence so is its localization
$$ \bS^{-1}(\bR/\bgp)\simeq \bS^{-1}\bR/\bS^{-1}\bgp ,$$
by Lemma \ref{a10Sep23}. Therefore, $\bS^{-1}\bgp\in \Spec (\bS^{-1}\bR)$.

(x) {\em $\gn_{\bS^{-1}\bR}\subseteq \bS^{-1}\gn_{\bR}$ and $\bS^{-1}\gn_{\bR}$ is an ideal of $\bS^{-1}\bR$}: By the statement (ix) and $|\min (\bR)|<\infty$, $$\gn_{\bS^{-1}\bR}=\bigcap_{\gq\in \Spec(\bS^{-1}\bR)}\gq\subseteq \bigcap_{\bgp\in \min (\bR)}\bS^{-1}\bgp=\bS^{-1}\bigg(\bigcap_{\bgp\in \min (\bR)}\bgp\bigg)=\bS^{-1}\gn_{\bR}$$
and the intersection of ideals of $\bS^{-1}\bR$, $\bS^{-1}\gn_{\bR}=\bigcap_{\bgp\in \min (\bR)}\bS^{-1}\bgp$,  is an ideal of $\bS^{-1}\bR$, see the statement (viii).

(xi)  $\bS^{-1}\bR/\bS^{-1}\gn_{\bR}\simeq
\widetilde{\bS}^{-1}\widetilde{R}$: By the statement (x), $\bS^{-1}\gn_{\bR}$ is an ideal of the ring $\bS^{-1}\bR$. Now, by the statement (iii), $\tbS\in \Den (\widetilde{R},0)$, and so 
$$\bS^{-1}\bR/\bS^{-1}\gn_{\bR}\simeq \bS^{-1}(\bR/\gn_{\bR})\simeq
\widetilde{\bS}^{-1}\widetilde{R}.
$$

(xii) {\em For all} $\bgp\in \min (\bR )$,  $\bR\cap \bS^{-1}\bgp=\bgp$: Clearly, $\bgp\subseteq \bR\cap \bS^{-1}\bgp$. Suppose that $\Big(\bR\cap \bS^{-1}\bgp\Big)\backslash \bgp\neq \emptyset$.  Fix an element $\br\in\Big( \bR\cap \bS^{-1}\bgp\Big)\backslash \bgp$. Then $\bs\br\in \bgp$ for some element $\bs\in \bS$. 
The element $\bs$ is a normal element of the ring $\bR$ and $\bs\not \in \gp$, by the statement (vi). Therefore, we have the inclusion of ideals $$(\bs)(\br)\subseteq  \bgp$$ which implies the inclusion $(\br)\subseteq \bgp$ (since $\bs\not\in \bgp$ and $\bgp$ is a prime ideal of the ring $\bR$), and so $\br\in \bgp$, a contradiction. 

(xiii) {\em The map 
$\min (\ga)\ra \Spec (\RSm)=\Spec (\bS^{-1}\bR)$, $\gp \mapsto \bS^{-1}\bgp $ is an injection}: 
 The result follows from the statements (ix) and (xii) (if $ \bS^{-1}\bgp = \bS^{-1}\bgp'$ for some $\gp, \gp'\in \min (\gq)$ then, by the statement (xii),  $$\bgp = \bR\cap \bS^{-1}\bgp = \bR\cap \bS^{-1}\bgp' =\bgp',$$
  and so $\gp = \gp'$). 
 
(xiv) {\em   For each $\gp \in \min (\ga )$, $\pi_\gp (S)\in \Den(R/\gp, 0)$}: By the statement (v) and (vii), $$\bS\cap \bgp=\emptyset\;\;{\rm and}\;\;\o_{\bs}(\bgp) = \bgp.$$ Hence, the subset $\pi_\gp(S)$ of $R/\gp$ is a multiplicative set that is generated by regular normal elements, and so $\pi_\gp (S)\in \Den(R/\gp, 0)$.

(xv) {\em For each $\gp \in \min (\ga )$, $\pi_\gp (S)^{-1}(R/\gp)\simeq \widetilde{\bS}^{-1}\widetilde{R}/\widetilde{\bS}^{-1}\widetilde{\bgp}\simeq \bS^{-1}\bR/\bS^{-1}\bgp$}:   By the statements  (iii),  (viii) and (xiv),
$$  \bS^{-1}\bR/\bS^{-1}\bgp\simeq \bS^{-1}(\bR/\bgp)\simeq \bS^{-1}(\widetilde{R}/\tbgp)\simeq \widetilde{\bS}^{-1}\widetilde{R}/\widetilde{\bS}^{-1}\widetilde{\bgp}\simeq \pi_\gp (S)^{-1}(R/\gp).
$$

The commutativity of the diagram is obvious.

4. (b) (xvi) $\gn_{\bS^{-1}\bR}=\bS^{-1}\gn_{\bR}$:
 By the assumption, the ideal $\gn_{\bR}=\bigcap_{\gp \in \min (\ga)}\bgp$ is a nilpotent ideal, i.e. $\gn_{\bR}^m=0$ for some natural number $m\geq 1$. Since the left ideal $\bS^{-1}\gn_{\bR}$ of the ring $\bS^{-1}\bR$ is an ideal (the statement (x)), 
$$\Big(\bS^{-1}\gn_{\bR}\Big)^m=\bS^{-1}\gn_{\bR}^m=0,$$
and so $\bS^{-1}\gn_{\bR}\subseteq \gn_{\bS^{-1}\bR}$. By the statement (x), we have the opposite inclusion, and the statement (xvi)  follows.

(xvii) $\bS^{-1}\bR/\bS^{-1}\gn_{\bR}\simeq
\bS^{-1}\bR/\gn_{\bS^{-1}}\bR$: The statement (xvii) follows from the statement (xvi).

The proof of the statement (b) is complete (in view of statement 3).

(a) (xviii) {\em The map $\min (\bR )\ra \min(\bS^{-1}\bR)$, $\bgp \mapsto \bS^{-1}\bgp $ is a bijection}:  By the statement  (xvi),
 $$\gn_{\bS^{-1}\bR}=\bS^{-1}\gn_{\bR}=\bS^{-1}\bigg(\bigcap_{\bgp\in \min (\bR)}\bgp\bigg)=\bigcap_{\bgp\in \min (\bR)}\bS^{-1}\bgp .
 $$
By the statements (vi) and (ix), the prime ideals 
$\{ \bS^{-1}\bgp\, | \, \bgp\in \min (\bR)\}$ of the ring $\bS^{-1}\bR$ are distinct. The second intersection above is irredundant since otherwise
$$\bS^{-1}\bgp'\subseteq \bigcap_{\bgp\in \min (\bR)\backslash \{ \bgp'\}}\bS^{-1}\bgp \subseteq \bgp.
 $$ for some $\bgp'\in \min (\bR)$. By the statement (xii),
 $$\bgp'=\bR\cap \bS^{-1}\bgp'\subseteq \bR\cap \bS^{-1}\bgp=\bgp,$$
a contradiction (since $\min (\bR)=\{ \bgp\, | \, \gp\in \min (\ga)\})$. Since the second intersection is iredundant, the statement (xviii) follows from Lemma \ref{b10Sep23}.
\end{proof}

Lemma \ref{a5Oct23} extends the results of Theorem \ref{A2Oct23} to a more general situation. 

\begin{lemma}\label{a5Oct23}
Let $R$ be a ring and $S'\in \Den_l(R, \ga)$. Suppose that for each element $s\in S'$ there is an element $t\in S'$ such that the element $ts$ is a normal element of $R$. Let $S$ be the set of all normal elements in $S'$. Then 

\begin{enumerate}
\item $S\in \Den_l(R, \ga)$ and $S^{-1}R\simeq S'^{-1}R$.
\item Theorem \ref{A2Oct23} holds for the denominator set $S'\in \Den_l(R, \ga)$.

\end{enumerate}
\end{lemma}

\begin{proof} 1. As a product of normal elements is a normal element and $S\subseteq S'$, the set $S$ is a multiplicative set of $R$. The set $S$ consists of normal elements. So, $S\in \Ore (R)$. Let $\pi : R\ra \bR :=R/\ga$, $r\mapsto \br := r+\ga$ and $\bS':=\pi (S')$. Then 
$$\bS':=\pi (S')\in \Den_l(\bR , 0)\;\; {\rm  and}\;\;S'^{-1}R\simeq \bS'^{-1}\bR.$$

(i) $\ass_l(S)=\ga$: Since $S\subseteq S'$, $\ass_l(S)\subseteq \ass_l(S')$. The reverse inclusion follows from the fact that  for each element $s\in S'$ there is an element $t\in S'$ such that the element $ts$ is a normal element of $R$. Therefore, $\ass_l(S)= \ass_l(S')=\ga$. 
 
(ii) $S\in \Den_l(R, \ga)$: In view of  the statement (i), it remains to show that if $rs=0$ for some elememts $s\in S$ and $r\in R$ then $tr=0$ for some element  $t\in S$. Since $S\subseteq S'$ and  $S'\in \Den_l(R, \ga)$, there is an element $s'\in S'$ such that $s'r=0$. Choose an element $t'\in S'$ such that the element $t:=t's'$ is a normal element of $R$, i.e. $t\in S$,  and so  $tr=t's'r=0$, as required.

(iii) $S^{-1}R\simeq S'^{-1}R$: By the statement (ii) and the inclusion $S\subseteq S'$, there is a monomorphism
$$S^{-1}R\ra S'^{-1}R, \;\; s^{-1}r\mapsto s^{-1}r.$$
By the assumption, for each element $s'\in S'$,  there is an element $t'\in S'$ such $t:=t's'$ is a normal element of the ring $R$, i.e. $t\in S$. Then for each $r\in R$, 
$$s'^{-1}r=t^{-1}t'r\in S^{-1}R,$$
i.e. the monomorphism is a surjection, and the statement (iii) follows. 

2. Statement 2 follows from statement 1 and Theorem \ref{A2Oct23}.
\end{proof}


\section{ Minimal primes of a semiprime ring and of its centre} \label{MINPZR} 

 In this section, proofs of Corollary  \ref{aB25Sep23} and Proposition \ref{B25Sep23} are given. Recall that Corollary  \ref{aB25Sep23} is a criterion for the  map $\rho_{R,\min}: \min (R)\ra \min (Z(R))$, $ \gp\mapsto \gp\cap Z(R)$ being a  well-defined map and  Proposition \ref{B25Sep23} is a criterion for the map $\rho_{R,\min}$  being  a surjection.\\

 Let $R$ be a ring and $Z(R)$ be its centre. For each prime ideal $\gq\in \Spec (Z(R))$, the field
$$ k(\gq):=Z(R)_\gq/\gq_\gq\simeq \Big( Z(R)/\gq\Big)_\gq $$
is called the {\bf residue field} of the prime ideal $\gq$. By applying the right exact functor $-\t_{Z(R)}R$ to the short exact sequence of $Z(R)$-modules 
$$0\ra \gq_\gq\ra Z(R)_\gq\ra k(\gq)\ra 0$$
we obtain an  exact sequence of right and left $R$-modules 
$$\gq_\gq\t_{Z(R)}R\ra Z(R)_\gq\t_{Z(R)}R\simeq R_\gq \ra k(\gq)\t_{Z(R)}R\ra 0.$$
Therefore, 
\begin{equation}\label{RqqkZ}
R_\gq/R_\gq \gq\simeq k(\gq)\t_{Z(R)}R
\end{equation}
 is an isomorphism of rings.  Let $\gq\in \Spec (Z(R))$ and $\gp \in \Spec (R)$. Then  $R/\gp$ is a prime ring. By Lemma  \ref{a10Sep23}, the ring $\Big(R/\gp\Big)_\gq\simeq R_\gq/\gp_\gq$ is a prime ring. Therefore, 
 the map 
\begin{equation}\label{RqqkZ-1}
\Spec (R)\ra \Spec(R_\gq), \;\;\gp \mapsto \gp_\gq
\end{equation}
  is a well-defined map.

\begin{proposition}\label{A25Sep23}
Let $R$ be a  ring and $Z(R)$ be its centre. Then 
\begin{enumerate}

\item Let  $\gq\in \Spec (Z(R))$. Then $\gq \in \im (\rho_R)$ iff $R_\gq\neq R_\gq\gq$.

\item For each $\gq\in \Spec (Z(R))$, the maps 
\begin{eqnarray*}
\{\gp\in \Spec (R)\, | \, \gp\cap Z(R)=\gq\}&\ra& V_{R_\gq}(R_\gq \gq)\ra \Spec (k(\gq)\t_{Z(R)}R),  \\
Q&\mapsto & \;\;\;\;\; Q_\gq \;\;\;\;\;\mapsto Q_\gq/R_\gq \gq\simeq k(\gq)\t_{Z(R)}Q,
\end{eqnarray*}
 are bijections where $V_{R_\gq}(R_\gq \gq):=\{ P\in \Spec(R_\gq)\, | \, R_\gq \gq \subseteq P\}$.

\item For each $\gq \in \im (\rho_R)$, there is a  minimal prime $\gp \in \min (R)$ such that $\gp\cap Z(R)=\gq$.

\end{enumerate}
\end{proposition}

\begin{proof} 2. By (\ref{RqqkZ-1}), the first map is a well-defined map. By the definition  of the localization at the prime ideal $\gq$ of the centre $Z(R)$ of $R$, the first map is a bijection. By (\ref{RqqkZ}), 
$R_\gq/R_\gq \gq\simeq k(\gq)\t_{Z(R)}R$, and so  the second map is a bijection.

1. Statements 1 follows  from statement 2.

3. Since $\gq \in \im (\rho_R)$,  there is a prime ideal $\gp'\in \Spec (R)$ such that $\gp' \cap Z(R)=\gq$. The prime ideal $\gp'$ contains a minimal prime ideal, say $\gp\in \min (R)$.  Notice that  $\gp\cap Z(R)\in \Spec (Z(R))$ and $\gp\cap Z(R)\subseteq \gp'\cap Z(R)=\gq $. Now,  by the minimality of $\gq$,   $\gp \cap Z(R)=\gq$, and the statement (ii) follows.  
\end{proof}

 In general, 
the restriction map $\rho_{R,{\rm min}}: \min (R)\ra \min (Z(R))$, $\gp\mapsto \gp\cap Z(R)$ is  not well-defined, see Lemma \ref{b29Sep23}.
 Let $K$ be a field, $n\in \{ 1,2, \ldots \}\cup \{ \infty\}$, $[n]:=\{i\in \N \, |\, 1\leq i\leq n\}$, $F_n=K\langle x_i\, | \, i\in [n]\rangle$ be a free $K$-algebra, $P_n=K[z_i\, | \, i\in [n]]$ be a polynomial  $K$-algebra and 
$$A_n=F_m\t_K P_n/(x_iz_i\, | \,  i\in  [n]).$$

\begin{lemma}\label{b29Sep23}
\begin{enumerate}
\item $\min (A_n)=\{ \gp_I\, | \,I\subseteq [n]\}$ where $\gp_I=(x_i, z_j)_{i\in I, j\in CI}$ and $CI =[n]\backslash I$.
\item For every $I\subseteq [n]$, $A_n/\gp_I\simeq F_{CI}\t_KP_I$ is a domain  where $F_{CI}=K\langle x_j\, | \, j\in CI\}$ and $P_I=K[z_j\, | \, j\in CI]$.
\item If $n=\infty$ then set $\min (A_\infty)$ is an uncountable set. 
\item $Z(A_n)=P_n$.
\item For every $I\subseteq [n]$, $\gp_I\cap Z(A_n)=(z_j)_{j\in CI}$.
 
\end{enumerate}
\end{lemma}

\begin{proof} 2, 4 and 5. Statements 2, 4 and 5 are  obvious.

1. Let $\gp$ be a minimal prime. For all $i\in [n]$,
$$\gp \supseteq \{ 0\} = (x_iz_i)=(x_i)(z_i).$$
By the minimality of $\gp$, precisely one of the elements $x_i$ or $z_i$ belongs to $\gp$, and statement 1 follows. 

3. Statement 3 follows from statement 1 and the fact that the set of all subsets of the set of natural numbers is an uncountable set.
 \end{proof}


\begin{lemma}\label{a25Sep23}
Let $R$ be a semiprime ring with $|\min (R)|<\infty$. Then  the centre $Z(R)$ of $R$ is a semiprime ring with $|\min (Z(R))\cap \im (\rho_R)|\leq |\min (R)|<\infty$.
\end{lemma}

\begin{proof}  Suppose that the  centre $Z(R)$ is not  a semiprime ring. Then it contains a nonzero nilpotent ideal, say $\ga$. Then $R\ga$ is a nonzero nilpotent ideal of the semiprime ring $R$, a contradiction. Therefore, the centre $Z(R)$ is a semiprime ring. 
By  Proposition \ref{A25Sep23}.(3), $|\min (Z(R))\cap \im (\rho_R)|\leq |\min (R)|<\infty$.
\end{proof}

\begin{proposition}\label{C25Sep23}
Let $R$ be a semiprime ring with $|\min (R)|<\infty$. Then the following statements  are equivalent: 
\begin{enumerate}
\item $\CC_{Z(R)}\not\subseteq \CC_R$.

\item  $\CC_{Z(R)}\cap \gp\neq \emptyset$  for some  $\gp \in \min (R)$.

\item The restriction map $\min (R)\ra \min (Z(R))$, $\gp\mapsto \gp\cap Z(R)$ is not a well-defined map. 
\end{enumerate}
\end{proposition}

\begin{proof} $(1\Rightarrow 2)$ Suppose that $\CC_{Z(R)}\not\subseteq \CC_R$. Then let us fix an element, say $c\in \CC_{Z(R)}\backslash  \CC_R$, i.e. the central element $c$ is a zero divisor in the ring $R$. Therefore, $\ga :=\ann_R (c)$ is nonzero annihilator ideal. It follows from the equality $(c) \ga=0$, that the ideal $(c)=Rc$ of $R$ is contained in the annihilator ideal $\ann_R (\ga)$. Then 
$$c\in \ann_R (\ga)\subseteq \gp\;\; {\rm for\; some}\;\; \gp \in \min (R),$$
(since the ring $R$ is a semiprime ring with $|\min(R)|<\infty$), i.e. $\CC_{Z(R)}\cap \gp\neq 0$.

$( 2\Rightarrow 3)$ Suppose that $\CC_{Z(R)}\cap \gp\neq \emptyset$  for some  $\gp \in \min (R)$.
Notice that  $$\gq :=Z(R)\cap \gp\in \Spec (Z(R)).$$ Then $\gq\not\in \min (Z(R))$ since otherwise $\gq\subseteq \CZ (R)$ where  $\CZ (R):=\bigcup_{\gl\in \min (Z(R))}\gl$ is the set of zero divisors of the semiprime commutative ring $Z(R)$ with $\min (Z(R)|<\infty$ but the non-empty   subset $\CC_{Z(R)}\cap \gp$ of $\gq$ has zero intersection with $\CZ (R)$, a contradiction. 

$(3\Rightarrow 2)$ Suppose that the restriction map $\min (R)\ra \min (Z(R))$, $\gp\mapsto \gp\cap Z(R)$ is not a well-defined map. So, $\gq :=\gp \cap Z(R)\not\in \min (Z(R))$ for some $\gp \in \min (R)$. Notice that  $\gq \in \Spec (Z(R))$. Therefore, 
$$\gq\not\subseteq \CZ (R)=\bigcup_{\gl\in \min (Z(R))}\gl$$
since otherwise $\gq \subseteq \gl$ for some $\gl \in \min (Z(R))$, i.e. $\gq=\gl\in \min (Z(R))$, a contradiction. Then $$\gp\cap \CC_{Z(R)}=\gq\cap \CC_{Z(R)}\neq \emptyset$$ since $\CC_{Z(R)}=Z(R)\backslash \CZ (R)$. 

$(2\Rightarrow 1)$ Suppose that $\CC_{Z(R)}\cap \gp\neq \emptyset$  for some  $\gp \in \min (R)$. Fix an element $c\in \CC_{Z(R)}\cap \gp$.  Recall that $\gp^c=\bigcap_{\gq\in \min (R)\backslash \{ \gp\}}\gq\neq 0$ and $\gp \gp^c=0$. Then 
$c\gp^c\subseteq \gp \gp^c=0$, and so $c\in\CC_{Z(R)}\backslash  \CC_R$.
\end{proof}

\begin{proof} {\bf (Proof of Proposition \ref{B25Sep23})}    The equivalences $(1\Leftrightarrow 2\Leftrightarrow 3)$ are equivalent to Corollary  \ref{aB25Sep23}.

$(1 \Rightarrow 4)$ Suppose that $\CC_{Z(R)}\subseteq \CC_R$. By Lemma \ref{a25Sep23}, the centre $Z(R)$ is a semiprime ring. By the assumption, $|\min (Z(R))|<\infty$. Since the centre $Z(R)$ is  a commutative   semiprime ring with $|\min (Z(R))|<\infty$, 
$$\CC_{Z(R)}=Z(R)\backslash \bigcup_{\gq\in \min (Z(R))}\gq.$$
Hence,  the quotient  ring $Q(Z(R))=\CC_{Z(R)}^{-1}Z(R)$ is a semiprime commutative Artinian ring with $$\min (Q(Z(R)))=\{\CC_{Z(R)}^{-1}\gq\, | \, \gq \in\min (R) \}.$$ Hence, the ring $Q(Z(R))$ is a semisimple Artinian  ring with $\Spec (Z(R))=\min (Z(R))$, and so   the quotient ring  $$Q(Z(R))=\CC_{Z(R)}^{-1}Z(R)\simeq \prod_{\gq\in \min (Z(R))}Z(R)_\gq =\prod_{\gq\in \min (Z(R))}k(\gq )$$ is a finite direct product of fields $Z(R)_\gq =k(\gq )$.  Since $\CC_{Z(R)}\subseteq \CC_R$,
$$\CC_{Z(R)}^{-1}Z(R)\subseteq \CC_{Z(R)}^{-1}Z(R)\simeq \prod_{\gq\in \min (Z(R))}Z(R)_\gq\t_{Z(R)}R\simeq \prod_{\gq\in \min (Z(R))}R_\gq,  $$
and so  the map $\rho_{R,\min}$ is a surjection, by Proposition \ref{A25Sep23} (since $R_\gq = R_\gq /R_\gq\gq$).

$(4\Rightarrow 3)$ The implication is obvious.
\end{proof}

\begin{corollary}\label{aC25Sep23}
Let $R$ be a semiprime ring such that  $|\min (R)|<\infty$,  $|\min (Z(R))|<\infty$ and $\CC_{Z(R)}\subseteq \CC_R$. Then 
\begin{enumerate}

\item  $R\subseteq \CC_{Z(R)}^{-1}R\simeq \prod_{\gq\in \min (Z(R))}R_\gq$.

\item  $Z(\CC_{Z(R)}^{-1}R)=\CC_{Z(R)}^{-1}Z(R)\simeq \prod_{\gq\in \min (Z(R))}Z(R)_\gq$ is a finite direct product of fields $Z(R)_\gq$.

\item  For every $\gq \in \min (Z(R))$,  $R_\gq$ is a semiprime ring, $\min (R_\gq)=\{\gp_\gq\, | \, \gp\in \min (R)\}$ and $|\min (R_\gq )|\leq |\min (R)|<\infty$.

\item  For every $\gq \in \min (Z(R))$,  $Z(R_\gq)\simeq Z(R)_\gq$,  i.e. the  $Z(R)_\gq$-algebra $R_\gq$ is a central algebra. 

\end{enumerate}
\end{corollary}

\begin{proof} 1. Statement 1 was proven in the  proof of Proposition \ref{B25Sep23}.

2.  By the assumption,  $\CC_{Z(R)}\subseteq \CC_R$. Hence, $R\subseteq \CC_{Z(R)}^{-1}R$. Let $c^{-1}r\in Z(\CC_{Z(R)}^{-1}R)$ for some elements $c\in \CC_{Z(R)}$ and $r\in R$. Then $r\in Z(R)$. It follows that $Z(\CC_{Z(R)}^{-1}R)=\CC_{Z(R)}^{-1}Z(R)$. The isomorphism 
 $\CC_{Z(R)}^{-1}Z(R)\simeq \prod_{\gq\in \min (Z(R))}Z(R)_\gq$ was proven in  in the  proof of Proposition \ref{B25Sep23}.

3.  Statement 3 follows  Theorem \ref{A10Sep23}.(1,2).

4.  Statement 4 follows from statements 1 and 3. 
 \end{proof}

{\bf Licence.} For the purpose of open access, the author has applied a Creative Commons Attribution (CC BY) licence to any Author Accepted Manuscript version arising from this submission.\\

{\bf Disclosure statement.} No potential conflict of interest was reported by the author.\\

{\bf Data availability statement.} Data sharing not applicable – no new data generated.

\small{

School of Mathematics and Statistics

University of Sheffield

Hicks Building

Sheffield S3 7RH

UK

email: v.bavula@sheffield.ac.uk}


\begin{thebibliography}{99}


 
 \bibitem{Bav-intdifline} V. V. Bavula,  The algebra of integro-differential operators on an affine line and its modules, {\em J. Pure Appl. Algebra} {\bf 217} (2013)  495-529. (Arxiv:math.RA: 1011.2997).
 
 
 

\bibitem{larglquot} V. V. Bavula,  The largest left quotient ring of a ring, {\it Comm. Algebra}  {\bf 44} (2016)
 no. 8, 3219-3261. (Arxiv:math.RA:1101.5107).
  
 \bibitem{LocSets} V. V. Bavula, Localizable sets and the localization of a ring at a localizable set,   {\em J. Algebra}, {\bf 60} (2022), no.15, 38--75.  (ArXiv:2112.13447).
 
\bibitem{Loc-groupUnits} V. V. Bavula, Localizations of a ring at localizable sets, their groups of units and saturations, {\it Math. Comp. Sci.}, {\bf 16} (2022), no. 1, Paper No. 10, 15 pp. 
 
 
 
\bibitem{Embed-SPrime-SSA} V. V. Bavula, Embeddings of semiprime  rings  into  semisimple Artinian  rings, submitted.
 
 \bibitem{GW} K. R. Goodearl and R. B. Warfield, Jr, An Introduction to Noncommutative Noetherian Rings, Second Edition, London Math. Soc. Student Texts {\bf 61}, Cambridge University Press, 2004.
 
  
\bibitem{Goldie-PLMS-1958} A. W.  Goldie, The structure of prime rings under ascending chain conditions. {\it Proc. London Math. Soc.}  (3) {\bf 8} (1958) 589--608.


\bibitem{Goldie-PLMS-1960} A. W.  Goldie,  Semi-prime rings with maximum condition. {\it Proc. London Math. Soc.}  (3) {\bf 10} (1960) 201--220.

\bibitem{Jategaonkar-LocNRings} A. V. Jategaonkar, Localization in Noetherian Rings, Londom Math. Soc. LMS 98, Cambridge Univ. Press, 1986.

\bibitem{Lam-Exbook} T. Y. Lam, Exercises in Modules and Ring, Springer, 2007.


\bibitem{Lesieur-Croisot-1959} L. Lesieur and R. Croisot, Sur les anneaux premiers noeth\'{e}riens \`{a} gauche. {\it Ann. Sci. \'{E}cole Norm. Sup.} (3) {\bf 76} (1959) 161--183.
 

\bibitem{MR} J. C. McConnell and J. C. Robson, Noncommutative Noetherian rings,
Wiley, Chichester,  1987.


\bibitem{Pierce-AssAlg} R. S. Pierce,  Associative algebras. Graduate Texts in Mathematics, 88. Studies in the History of Modern Science, 9. Springer-Verlag, New York-Berlin, 1982. xii+436 pp.

\bibitem{Stenstrom-RingQuot} B. Stenstr\"{o}m, Rings of Quotients, Springer-Verlag, Berlin, Heidelberg, New York, 1975.










\end{thebibliography}
\end{document}